\documentclass[a4paper,12pt]{amsart}

\newcommand{\varphibar}{{\overline{\varphi}}}
\newcommand{\phibar}{{\overline{\phi}}}
\newcommand{\Phibar}{{\overline{\Phi}}}

\usepackage{color}
\newcommand{\pair}[1]{{\langle {#1} \rangle}}
\newcommand{\floor}[1]{{\lfloor {#1} \rfloor}}
\newcommand{\reduced}{\mathrm{red}}

\usepackage{comment}

\usepackage{latexsym}
\usepackage{amscd}
\usepackage{amsmath}
\usepackage{amssymb}

\usepackage{amsthm}
\usepackage{float}
\usepackage{graphicx}
\usepackage{psfrag}
\usepackage{pst-all}

\theoremstyle{plain}
\newtheorem{theorem}{Theorem}[section]
\newtheorem{corollary}[theorem]{Corollary}
\newtheorem{lemma}[theorem]{Lemma}

\newtheorem{proposition}[theorem]{Proposition}

\theoremstyle{definition}

\newtheorem{remark}[theorem]{Remark}

\numberwithin{equation}{section}
\numberwithin{figure}{section}

\newcommand{\la}{\left\langle}
\newcommand{\ra}{\right\rangle}
\newcommand{\lb}{\left(}
\newcommand{\rb}{\right)}
\newcommand{\lc}{\left\{}
\newcommand{\rc}{\right\}}

\newdimen\argwidth
\def\db[#1\db]{
 \setbox0=\hbox{$#1$}\argwidth=\wd0
 \setbox0=\hbox{$\left[\box0\right]$}
  \advance\argwidth by -\wd0
 \left[\kern.3\argwidth\box0 \kern.3\argwidth\right]}

\newcommand{\I}{\ensuremath{\sqrt{-1}}}
\renewcommand{\Re}{\ensuremath{\mathfrak{Re}}}
\renewcommand{\Im}{\ensuremath{\mathop{\mathfrak{Im}}\nolimits}}

\newcommand{\Ker}{\operatorname{Ker}}

\newcommand{\Res}{\mathop{\mathrm{Res}}}

\newcommand{\id}{\ensuremath{\mathop{\mathrm{id}}}}
\newcommand{\Sym}{\operatorname{Sym}}

\newcommand{\SL}{\operatorname{SL}}

\newcommand{\coh}{\operatorname{coh}}

\newcommand{\Ext}{\mathop{\mathrm{Ext}}\nolimits}
\newcommand{\Hom}{\mathop{\mathrm{Hom}}\nolimits}

\newcommand{\Lotimes}{\stackrel{\bL}{\otimes}}

\newcommand{\Proj}{\operatorname{Proj}}

\newcommand{\lto}{\longrightarrow}

\newcommand{\bC}{\ensuremath{\mathbb{C}}}

\newcommand{\bG}{\ensuremath{\mathbb{G}}}

\newcommand{\bL}{\ensuremath{\mathbb{L}}}

\newcommand{\bN}{\ensuremath{\mathbb{N}}}

\newcommand{\bP}{\ensuremath{\mathbb{P}}}
\newcommand{\bQ}{\ensuremath{\mathbb{Q}}}
\newcommand{\bR}{\ensuremath{\mathbb{R}}}

\newcommand{\bT}{\ensuremath{\mathbb{T}}}

\newcommand{\bZ}{\ensuremath{\mathbb{Z}}}

\newcommand{\be}{\ensuremath{\boldsymbol{e}}}
\newcommand{\bd}{\ensuremath{\boldsymbol{d}}}

\newcommand{\bq}{\ensuremath{\boldsymbol{q}}}

\newcommand{\bv}{\ensuremath{\boldsymbol{v}}}

\newcommand{\scA}{\ensuremath{\mathcal{A}}}
\newcommand{\scB}{\ensuremath{\mathcal{B}}}
\newcommand{\scC}{\ensuremath{\mathcal{C}}}
\newcommand{\scD}{\ensuremath{\mathcal{D}}}
\newcommand{\scE}{\ensuremath{\mathcal{E}}}
\newcommand{\scF}{\ensuremath{\mathcal{F}}}
\newcommand{\scG}{\ensuremath{\mathcal{G}}}
\newcommand{\scH}{\ensuremath{\mathcal{H}}}
\newcommand{\scI}{\ensuremath{\mathcal{I}}}
\newcommand{\scJ}{\ensuremath{\mathcal{J}}}
\newcommand{\scK}{\ensuremath{\mathcal{K}}}
\newcommand{\scL}{\ensuremath{\mathcal{L}}}
\newcommand{\scM}{\ensuremath{\mathcal{M}}}
\newcommand{\scN}{\ensuremath{\mathcal{N}}}
\newcommand{\scO}{\ensuremath{\mathcal{O}}}
\newcommand{\scP}{\ensuremath{\mathcal{P}}}
\newcommand{\scQ}{\ensuremath{\mathcal{Q}}}
\newcommand{\scR}{\ensuremath{\mathcal{R}}}
\newcommand{\scS}{\ensuremath{\mathcal{S}}}
\newcommand{\scT}{\ensuremath{\mathcal{T}}}
\newcommand{\scU}{\ensuremath{\mathcal{U}}}
\newcommand{\scV}{\ensuremath{\mathcal{V}}}
\newcommand{\scW}{\ensuremath{\mathcal{W}}}
\newcommand{\scX}{\ensuremath{\mathcal{X}}}
\newcommand{\scY}{\ensuremath{\mathcal{Y}}}
\newcommand{\scZ}{\ensuremath{\mathcal{Z}}}

\newcommand{\frakM}{\ensuremath{\mathfrak{M}}}
\newcommand{\frakN}{\ensuremath{\mathfrak{N}}}

\newcommand{\frakX}{\ensuremath{\mathfrak{X}}}
\newcommand{\frakY}{\ensuremath{\mathfrak{Y}}}
\newcommand{\frakZ}{\ensuremath{\mathfrak{Z}}}

\newcommand{\frakm}{\ensuremath{\mathfrak{m}}}

\newcommand{\bCx}{\bC^{\times}}

\newcommand{\h}{h}
\newcommand{\X}{\boldsymbol{X}}
\newcommand{\Y}{\boldsymbol{Y}}
\newcommand{\Xbar}{\overline{X}}

\newcommand{\scEtilde}{\scE}
\newcommand{\scFtilde}{\scF}

\newcommand{\Sigmatilde}{\widetilde{\Sigma}}

\newcommand{\scEbar}{\overline{\scE}}
\newcommand{\scFbar}{\overline{\scF}}

\newcommand{\Horb}{H_{\mathrm{orb}}}
\newcommand{\Itilde}{\widetilde{I}}

\newcommand{\Hqd}{H_{\bq, \bd}}
\newcommand{\Qred}{{Q^{\reduced}}}

\title{Invariants of hypergeometric groups for
Calabi-Yau complete intersections in weighted projective spaces}
\author{Susumu Tanab\'{e}
and
Kazushi Ueda%
}
\date{}
\begin{document}

\maketitle

\begin{abstract}
Let $Y$ be a Calabi-Yau complete intersection
in a weighted projective space.
We show that the space of quadratic invariants of
the hypergeometric group
associated with the twisted $I$-function
%of the pair $(\bP, Y)$
is one-dimensional,
and spanned by the Gram matrix of a split-generator
of the derived category of coherent sheaves on $Y$
with respect to the Euler form.
\end{abstract}

\section{Introduction} \label{sc:introduction}

Let $\bq = (q_0, \dots, q_N)$ and $\bd = (d_1, \dots, d_r)$
be sequences of positive integers
such that
$$
 Q := q_0 + \dots + q_N = d_1 + \dots + d_r,
$$
and consider a complete intersection $Y$
of degree $(d_1, \dots, d_r)$
in the weighted projective space
$\bP = \bP(q_0, \dots, q_N)$.
If $Y$ is smooth,
then it is a Calabi-Yau manifold of dimension
$
 n = N - r.
$
%We assume $N - r \geq 1$ in this paper.
The derived category $D^b \coh \bP$ of coherent sheaves on $\bP$
has a full strong exceptional collection
$$
 (\scEtilde_i)_{i=1}^Q
   = (\scO_{\bP}, \scO_{\bP}(1), \dots, \scO_{\bP}(Q-1))
$$
of line bundles \cite{Beilinson, Auroux-Katzarkov-Orlov_WPP}.
Let $(\scFtilde_i)_{i=1}^Q$ be the full exceptional collection
dual to $(\scEtilde_i)_{i=1}^Q$,
so that
$$
 \chi(\scEtilde_{Q-i+1}, \scFtilde_j)
  = \delta_{ij}
$$
where
\begin{equation} \label{eq:Euler}
 \chi(\scE, \scF)
  = \sum_k (-1)^k \dim \Ext^k (\scE, \scF)
\end{equation}
is the Euler form.
The derived restrictions
$
\{ \scFbar_i \}_{i=1}^Q
$
of $\{ \scFtilde_i \}_{i=1}^Q$ to $Y$
split-generate the derived category $D^b \coh Y$
of coherent sheaves on $Y$
%by shifts, cones and direct summands
\cite[Lemma 5.4]{Seidel_K3}.

Following \cite[Equation (4)]{Coates-Corti-Lee-Tseng},
let us introduce the {\em twisted $I$-function}
\begin{equation} \label{eq:Gamma-series}
 I_{\bP, Y}(t) = \sum_\alpha
 e^{P_\alpha \log t} t^{\rho_\alpha}
 \sum_{n = 0}^\infty
  t^n
   \frac{
    \prod_{k = 1}^r
    \prod_{\substack{b \colon \pair{b} = \pair{\rho_\alpha d_k} \\
                          0 < b \le (n + \rho_\alpha) d_k}}
          (d_k P_\alpha + b)
    }
    {
    \prod_{\nu = 0}^N
    \prod_{\substack{b \colon \pair{b} = \pair{\rho_\alpha q_\nu} \\
                          0 < b \le (n + \rho_\alpha) q_\nu}}
          (q_\nu P_\alpha + b)
    },
\end{equation}
where $\pair{x} = x - \floor{x}$ is the fractional part of $x$.
This is an element of the ring
$
 \bigoplus_\alpha \bC[P_\alpha]_\alpha
  / (P_\alpha^{\mu_\alpha})_\alpha,
$
which is isomorphic to the orbifold cohomology of $\bP$
as a vector space.
See Section 2 for the definition of $\rho_\alpha \in \bQ$ and
$\mu_\alpha \in \bN$.
The twisted $I$-function $I_{\bP, Y}$ contains
the information of twisted Gromov-Witten invariant
of the bundle
$
 \scO_{\bP}(d_1) \oplus \cdots \oplus \scO_{\bP}(d_r)
  \to \bP
$
through the twisted $J$-function
\cite{Coates-Corti-Iritani-Tseng_CGZTGWI}.

The components of the twisted $I$-function
span the space of solutions
of the hypergeometric differential equation
\begin{align} \label{eq:hge}
 \left[
  \prod_{\nu=0}^N \prod_{a=0}^{q_\nu - 1}
   \left( q_\nu \theta_t - a \right)
  - t
     \prod_{k=1}^r \prod_{b=1}^{d_k}
       \left(
        d_k \theta_t + b
       \right)
 \right] I = 0
\end{align}
%around the origin
where $\theta_t = t \dfrac{\partial}{\partial t}$.
Let $\scH$ be the differential operator
on the left hand side and
$\scH^{\reduced}$ be the operator obtained from $\scH$
by removing common factors from the two summands of $\scH$.
Both $\scH$ and $\scH^{\reduced}$ have regular singularities
at $0$, $\infty$ and
$
 \lambda = \left. \prod_{\nu=0}^N q_\nu^{q_\nu}
      \right/ \prod_{k=1}^r d_k^{d_k}.
$
The local system $\scL^{\reduced}$ defined by $\scH^{\reduced}$
is irreducible, and its rank $Q^{\reduced}$ is smaller
than the rank $Q$ of the local system $\scL$
defined by $\scH$.
The irreducible system $\scL^{\reduced}$ supports
a pure and polarized variation of Hodge structures,
whose Hodge numbers are computed by Corti and Golyshev
\cite[Theorem 1.3]{Corti-Golyshev}.

The mirror of $Y$ is identified
by Batyrev and Borisov \cite{Batyrev-Borisov_CYCITV}
as the family of toric complete intersections
whose affine part is given by
\begin{equation} \label{eq:mirror}
 X_t = \{ (x_0, \cdots, x_N) \in (\bCx)^{N+1} \mid
  f_i (x) = 0, \ i = 0, \ldots, r \}
\end{equation}
where
\begin{align*}
 f_0 (x) &= x_0^{q_0} x_1^{q_1} \dots x_N^{q_N} - t, \\
 f_i(x) &= \sum_{k \in S_i} x_k - 1, \qquad \qquad
  i = 1, \ldots, r,
\end{align*}
the variable $t$ is the parameter of the family,
and
$
 S_1 \sqcup \dots \sqcup S_r
  = \{ 0, 1, \dots, N \}
$
is a partition of $\{ 0, 1, \dots, N \}$ into $r$ disjoint subsets
such that
$
 d_k = \sum_{i \in S_k} q_i.
$
The period integral
\begin{equation} \label{eq:period}
 I(t)
  = \int_\gamma
     \frac{x_0^{q_0} \dots x_N^{q_N}}
          {d f_0 \wedge \dots \wedge d f_r}
     \frac{d x_0}{x_0} \wedge \dots \wedge \frac{d x_n}{x_n}
\end{equation}
%of the holomorphic volume form on $X_t$
for a middle-dimensional cycle $\gamma \in H_{n}(X_t)$
%$\bq = (q_0, q_2, \cdots, q_N)$ and
%$\bone = (1, 1, \cdots, 1)$
satisfies the irreducible hypergeometric differential equation
$\scH^{\reduced} I = 0$.

Define the {\em hypergeometric group}
$\Hqd$
as the subgroup of $GL(Q, \bZ)$
generated by
\begin{equation} \label{eq:hinfty}
h_\infty
   = \begin{pmatrix}
      0 & 0 & \dots & 0 & - A_Q \\
      1 & 0 & \dots & 0 & - A_{Q-1} \\
      0 & 1 & \dots & 0 & - A_{Q-2} \\
      \vdots & \vdots & \ddots & \vdots & \vdots \\
      0 & 0 & \dots & 1 & - A_1
     \end{pmatrix}  
\end{equation}
and
\begin{equation} \label{eq:h0}
 h_0^{-1}
   = \begin{pmatrix}
      0 & 0 & \dots & 0 & - B_Q \\
      1 & 0 & \dots & 0 & - B_{Q-1} \\
      0 & 1 & \dots & 0 & - B_{Q-2} \\
      \vdots & \vdots & \ddots & \vdots & \vdots \\
      0 & 0 & \dots & 1 & - B_1
     \end{pmatrix},
%\end{align*}
\end{equation}
where
%$(A_1, \cdots, A_{Q})$ and $(B_1, \cdots, B_{Q})$
%are defined by
\begin{equation} \label{eq:a}
 \prod_{k=1}^r (T^{d_k} - 1)
  = T^{Q} + A_1 T^{Q-1} + A_2 T^{Q-2} + \dots + A_{Q}
\end{equation}
and
\begin{equation} \label{eq:b}
 \prod_{\nu=0}^N (T^{q_\nu} - 1)
  = T^{Q} + B_1 T^{Q-1} + B_2 T^{Q-2} + \dots + B_{Q}
\end{equation}
are the characteristic polynomials
of the monodromy of \eqref{eq:hge}
at infinity and zero respectively.
%The system \eqref{eq:hge} has another singularity at
%$
% t
%%  = \lambda
%  := \prod_{\nu=0}^N q_\nu^{q_\nu}
%   \left/ \prod_{k=1}^r d_k^{d_k} \right.
%$
%with the monodromy matrix $h_1$ satisfying
%$
% h_0 \cdot h_1 \cdot h_\infty = 1.
%$
If the system
%of \eqref{eq:hge}
is irreducible,
then a result of Levelt \cite{Levelt} states
that the monodromy group is conjugate
to the hypergeometric group $\Hqd$.
Although the system $\scL$ is reducible and
one can not apply the result of Levelt directly,
we can show the following:

\begin{theorem} \label{th:monodromy-group}
For any sequences
$\bq = (q_0, \dots, q_N)$ and
$\bd = (d_1, \dots, d_r)$
of positive integers
satisfying
$$
 Q := q_0 + \dots + q_N = d_1 + \dots + d_r,
$$
the monodromy group of \eqref{eq:hge} is conjugate
to the hypergeometric group
$\Hqd$
\end{theorem}

%A result of Levelt
%recalled in section \ref{sc:hg}
%allows us to write down a set of generators
%of the hypergeometric group explicitly
%under a suitable choice of a basis
%of the space of solutions.
%Let $\Hqd$ be the hypergeometric group
%associated with \eqref{eq:giventalhge},
%considered as a subgroup of $\GL(Q, \bC)$
%by this basis.
%One can easily see that
%$\Hqd$
%is a subgroup of $\GL(Q, \bZ)$
%in this case.

An element $h \in \Hqd$ acts naturally
on the space of $Q \times Q$-matrices by
$$
 \Hqd \ni h : 
  X \mapsto h \cdot X \cdot h^T,
$$
where $h^T$ is the transpose of $h$.
The following is a corollary of Theorem \ref{th:monodromy-group}:
%We prove the following in sections \ref{sc:ihg} and \ref{sc:CYci}:

\begin{theorem} \label{th:invariant}
The space of $Q \times Q$-matrices
invariant under the action of the monodromy group
$\Hqd$ of \eqref{eq:hge}
is one-dimensional and
spanned by the Gram matrix
$$
 \left( \chi(\scFbar_i, \scFbar_j) \right)_{i, j = 1}^Q
$$
of the split-generator $\{ \scFbar_i \}_{i=1}^Q$
with respect to the Euler form.
\end{theorem}

Theorem \ref{th:invariant} is closely related
to the works of Horja \cite[Theorem 4.9]{Horja_HFMSTV}
and Golyshev \cite[\S 3.5]{Golyshev_RRV},
which goes back to Kontsevich
\cite{Kontsevich_ENS98}.
The main difference from their works is that
we work with the reducible system $\scL$
which contains sections
not coming from period integrals on the mirror manifold.
In the case of the irreducible local system $\scL^\reduced$,
Golyshev gave a beautiful interpretation
%of $\scL^\reduced$
in terms of autoequivalences of the derived category
of the mirror manifold $Y$
(cf. Section \ref{sc:stokes}).
%We discuss the construction of Golyshev in Section \ref{sc:stokes}.
Although the geometric meaning
of extra sections of $\scL$
%not coming from period integrals
is unclear,
%it is interesting to note that
our proof of Theorem \ref{th:invariant} also use
autoequivalences of the derived category
of the mirror manifold $Y$,
just as in the case of the irreducible system.

The organization of this paper is as follows:
The proof of Theorem \ref{th:monodromy-group} is given
in Section \ref{sc:mhg}.
The essential step is to show the existence
of a cyclic vector for the monodromy around the origin,
which satisfies additional condition
with respect to the monodromy at infinity.
The uniqueness of the invariant of the hypergeometric group
is shown in Section \ref{sc:ihg},
and the invariance of the Gram matrix
of the split-generator
with respect to the Euler form is shown in Section \ref{sc:CYci}.
In Section \ref{sc:stokes},
we discuss the relationship
between the Gram matrix in Theorem \ref{th:invariant}
%for a hypersurface
and the Stokes matrix
for the quantum cohomology of the weighted projective space.
%Although the Stokes matrix is a $Q \times Q$ matrix,
%it is completely determined by the monodromy
%of the irreducible system $\scL^\reduced$ of rank $Q^\reduced$,
%whose monodromy is known by Horja
%\cite{Horja_HFMSTV} and Golyshev \cite{Golyshev_RRV}.

{\bf Acknowledgment:}
We thank Hiroshi Iritani for valuable discussions,
and the anonymous referee for suggesting improvements.
This work is supported by Grant-in-Aid for Scientific Research
(No.~20540086 and No.~20740037).

\section{Monodromy of  hypergeometric equation} \label{sc:mhg}

We prove Theorem \ref{th:monodromy-group} in this section.
Let $h_0$, $h_1$ and $h_{\infty}$
be the global monodromy matrix
of the hypergeometric differential equation \eqref{eq:hge}
around the origin, one and infinity
with respect to some basis of solutions
satisfying
$
 h_0 \cdot h_1 \cdot h_\infty = 1.
$
Recall that a vector $v \in \bC^Q$ is said to be {\em cyclic}
with respect to $h \in GL(Q, \bC)$ if the set
$\{ h^i \cdot v \}_{i=0}^{Q-1}$ spans $\bC^Q$.
The following lemma is used by Levelt \cite{Levelt}
to compute the monodromy of hypergeometric functions
(see also Beukers and Heckman
\cite[Theorem 3.5]{Beukers-Heckman}).
\begin{lemma}
Assume that there exists a vector satisfying
% \begin{align} \label{eq:condition}
%  h_0^i v &= h_\infty^{-i} v, & i &= 0, 1, \dots, Q-2,
% \end{align}
\begin{align} \label{eq:condition}
 h_0^i v = h_\infty^{-i} v, \qquad i = 1, \dots, Q-1,
\end{align}
which is cyclic with respect to $h_0$.
Then the monodromy group of \eqref{eq:hge} is isomorphic to
$\Hqd$.
\end{lemma}

\begin{proof}
The condition \eqref{eq:condition} shows that
the action of $h_0$ and $h_{\infty}^{-1}$
with respect to the basis
$\{ h_{\infty}^{-i} v \}_{i = 0}^{Q-1}$ of $\bC^Q$
is given by
$$
 \begin{pmatrix}
  0 & 0 & \dots & 0 & * \\
  1 & 0 & \dots & 0 & * \\
  0 & 1 & \dots & 0 & * \\
  \vdots & \vdots & \ddots & \vdots & \vdots \\
  0 & 0 & \dots & 1 & * \\
 \end{pmatrix}.
$$
The last line is determined by the characteristic equations
$$
 \det (T - h_0)
  = T^Q + A_1 T^{Q-1} + A_2 T^{Q-2} + \dots + A_Q
$$
and
$$
 \det (T - h_{\infty}^{-1})
  = T^Q + B_1 T^{Q-1} + B_2 T^{Q-2} + \dots + B_Q.
$$
\end{proof}

\begin{remark}
Even if there is no vector satisfying \eqref{eq:condition}
which is cyclic with respect to $h_0$,
one can consider the subspace generated from any vector
satisfying \eqref{eq:condition} by the action of $h_0$,
and the resulting matrix presentation of the monodromy action
with respect to $\{ h_{\infty}^{-i} v \}_{i = 0}^{Q-1}$
will be given by \eqref{eq:h0} and \eqref{eq:hinfty}.
Since $\{ h_{\infty}^{-i} v \}_{i = 0}^{Q-1}$ is not a basis
but only a generator in such a case,
this matrix presentation is not unique.
\end{remark}

Hence the proof of Theorem \ref{th:monodromy-group}
is reduced to the following:

\begin{proposition} \label{prop:main}
There exists a vector $v$
in the space of solutions of \eqref{eq:hge}
which is cyclic with respect to $h_0$
and satisfies \eqref{eq:condition}.
% \begin{equation} %\label{eq:condition}
%  h_0^i v = h_\infty^{-i} v, \qquad, i = 0, 1, \dots, Q-2,
% \end{equation}
\end{proposition}

The rest of this section is devoted to the proof of
Proposition \ref{prop:main}.
Note that the existence of such $v$ is automatic
if the hypergeometric differential equation is irreducible
(cf.~\cite[Theorem 3.5]{Beukers-Heckman}).

The hypergeometric differential equation \eqref{eq:hge}
has regular singularities
at $t = 0, \infty$ and $\lambda$
where
$
 \lambda = \left. \prod_{\nu=0}^N q_\nu^{q_\nu}
  \right/ \prod_{k=1}^r d_k^{d_k}.
$
To simplify notations,
we introduce another variable $z$ by
$
 t = \lambda z.
$
Then the local exponents are given by
\begin{align*}
 &\frac{b}{d_k},
  \qquad k = 1, \dots, r,
  \quad  b = 1, \dots, d_k &
  \qquad & & \text{at } z &= \infty, \\
 &\frac{a}{q_\nu},
  \qquad \nu = 1, \dots, N,
  \quad  a   = 0, \dots, q_\nu - 1 &
  \qquad & & \text{at } z &= 0, \text{ and} \\
 0, \ 1, \ 2, &\dots, Q-2, \ 
  \frac{n - 1}{2} &
  \qquad & & \text{at } z &= 1.
%   &= \frac{\prod_{i=0}^N q_i^{q_i}}{\prod_{k=1}^r d_k^{d_k}}.
%   &= \left. \prod_{i=0}^N q_i^{q_i} \right/ \prod_{k=1}^r d_k^{d_k}.
\end{align*}
Let
$$
 1 > \rho_1 > \rho_2 > \dots > \rho_p = 0
$$
be the characteristic exponents of \eqref{eq:hge} at $z = 0$
so that
$$
 \{ \rho_1, \cdots, \rho_p \}
  = \bigcup_{0 \le \nu \le N}
     \left\{ 0, \frac{1}{q_\nu}, \dots, \frac{q_\nu - 1}{q_\nu} \right\}.
$$
Let further
$$
 \mu_\alpha
  = \# \left\{ (q_\nu, a) \left|
            \rho_\alpha = \frac{a}{q_\nu}, \quad
            0 \le a \le q_\nu - 1, \quad
            0 \le \nu \le N \right. \right\}
$$
be the multiplicity of the exponent $\rho_\alpha$
and put
$$
 e_\alpha = \exp(2 \pi \I \rho_\alpha), \qquad 1 \le \alpha \le p.
$$ We remark that $\mu_p=N+1.$ For the quantity defined by 
$$ \nu_\alpha
  = \# \left\{ (d_k, b) \left|
            \rho_\alpha =  1-\frac{b}{d_k}, \quad
            1 \le b \le d_k - 1, \quad
            1 \le k \le r \right. \right\}, \qquad 1 \le \alpha \le p$$
the following relation holds
$$
 Q^{\reduced}= \sum_{\alpha=1}^p (\mu_\alpha- \nu_\alpha).
$$ 

Let us introduce the matrices
$$
 M_0 =
  \begin{pmatrix}
   \rho_1 \id_{\mu_1} + J_{\mu_1, -} & 0 & \cdots & 0 \\
   0 & \rho_2 \id_{\mu_2} + J_{\mu_2, -} & \cdots & 0 \\
   \vdots & \vdots & \ddots & \vdots  \\
   0 & 0 & \cdots & \rho_p \id_{\mu_p} + J_{\mu_p, -}
  \end{pmatrix}
$$
$$
 E_0 =
  \begin{pmatrix}
   e_1 \id_{\mu_1} + J_{\mu_1, -} & 0 & \cdots & 0 \\
   0 & e_2 \id_{\mu_2} + J_{\mu_2, -} & \cdots & 0 \\
   \vdots & \vdots & \ddots & \vdots  \\
   0 & 0 & \cdots & e_p \id_{\mu_p} + J_{\mu_p, -}
  \end{pmatrix}
$$
where $J_{i, \pm}$ are $i \times i$ matrices defined by
$$
 J_{i, +} =
 \begin{pmatrix}
   0 & 1 & 0 & \cdots & 0 \\
   0 & 0 & 1 & \cdots & 0 \\
   \vdots & \vdots & \vdots & \ddots & \vdots \\
   0 & 0 & 0 & \cdots & 1 \\
   0 & 0 & 0 & \cdots & 0
 \end{pmatrix}
$$
and
$$
 J_{i, -} =
 \begin{pmatrix}
  0 & 0 & \cdots & 0 & 0 \\
  1 & 0 & \cdots & 0 & 0 \\
  0 & 1 & \cdots & 0 & 0 \\
  \vdots & \vdots & \ddots & \vdots & \vdots \\
  0 & 0 & \cdots & 1 & 0
 \end{pmatrix}.
$$
% \begin{align*}
%  J_{i, +} &=
%  \begin{pmatrix}
%    0 & 1 & 0 & \cdots & 0 & 0 \\
%    0 & 0 & 1 & \cdots & 0 & 0 \\
%    0 & 0 & 0 & \cdots & 0 & 0 \\
%    \vdots & \vdots & \vdots & \ddots & \vdots  &\vdots \\
%    0 & 0 & 0 & \cdots & 0 & 1 \\
%    0 & 0 & 0 & \cdots & 0 & 0
%  \end{pmatrix}, &
%  J_{i, -} &=
%  \begin{pmatrix}
%   0 & 0 & \cdots & 0 & 0 \\
%   1 & 0 & \cdots & 0 & 0 \\
%   0 & 1 & \cdots & 0 & 0 \\
%   \vdots & \vdots & \ddots & \vdots & \vdots \\
%   0 & 0 & \cdots & 1 & 0
%  \end{pmatrix}.
% \end{align*}

A series solution to \eqref{eq:hge} at the origin can be obtained
by the Frobenius method:

\begin{lemma} \label{lm:Gamma-series}
A basis of solutions to \eqref{eq:hge} can be obtained
as the coefficient of $P_\alpha^i$
for $\alpha = 1, \ldots, p$ and $i = 0, \dots, \mu_\alpha - 1$
in the $\Gamma$-series
in \eqref{eq:Gamma-series}.
Solutions to the irreducible equation
$H^{red} u=0$ correspond to the coefficient of  $P_\alpha^i$
for $i = 0, \dots, \mu_\alpha - \nu_\alpha-1$
in  \eqref{eq:Gamma-series}. 
\end{lemma}

\begin{proof}
Let
\begin{align} \label{eq:Gamma-series2}
 H
  = e^{P_\alpha \log t} t^{\rho_\alpha}
 \sum_{n = 0}^\infty
  t^n
   \frac{
    \prod_{k = 1}^r
    \prod_{\substack{b \colon \pair{b} = \pair{\rho_\alpha d_k} \\
                          0 < b \le (n + \rho_\alpha) d_k}}
          (d_k P_\alpha + b)
    }
    {
    \prod_{\nu = 0}^N
    \prod_{\substack{b \colon \pair{b} = \pair{\rho_\alpha q_\nu} \\
                          0 < b \le (n + \rho_\alpha) q_\nu}}
          (q_\nu P_\alpha + b)
    }
\end{align}
be the $\Gamma$-series in \eqref{eq:Gamma-series}
considered as a formal power series
in $P_\alpha$, $t$ and $\log t$.
Here $e^{P_\alpha \log t}$ is considered as
$
 \sum_{n=0}^\infty (P_\alpha \log t)^n / (n!),
$
which satisfies
$
 \theta \lb e^{P_\alpha \log t} \rb = P_\alpha e^{P_\alpha \log t}.
$
%(the element $t^\gamma$ is regarded as $e^{\gamma \log t}$).
A direct calculation shows
\begin{align*}
\left[
  \prod_{\nu=0}^N
  \prod_{a=0}^{q_\nu - 1} (q_\nu \theta - a)
  - 
 t \prod_{k=1}^r
  \prod_{b=1}^{d_k} (d_k \theta + b)
\right] I
 &= \left(
    \prod_{\nu = 0}^N
    \prod_{a=0}^{q_\nu-1}
          (q_\nu P_\alpha + q_\nu \rho_\nu - a)
 \right)
 e^{P_\alpha  \log t} t^{\rho_\alpha},
\end{align*}
where the right hand side is proportional to
$P_\alpha^{\mu_\alpha}$,
so that the coefficients of $P_\alpha^i$
for $i = 0, 1, \ldots, \mu_\alpha - 1$ give solutions
to the hypergeometric equation \eqref{eq:hge}.
If one think of the series in \eqref{eq:Gamma-series2}
as a formal power series in $t$ and $\log t$
with values in the $\mu_\alpha$-dimensional vector space
$\bC[P_\alpha] / (P_\alpha^{\mu_\alpha})$,
then it is a polynomial in $t^{\rho_\alpha}$ and $\log t$
(this is clear from \eqref{eq:Gamma-series2})
and a convergent power series in $t$
(this follows either from the fact that the origin is a regular singularity of \eqref{eq:hge},
or a direct estimate of the radius of convergence,
which gives $\lambda$),
which gives a multi-valued solution to \eqref{eq:hge}.

Alternatively,
one can also argue as follows:
%An alternative proof can be obtained as follows:
For the set of poles $\Pi_\alpha=\{w \in \bC: |w +\rho_\alpha + n| =\epsilon, n \in \bN, 0<\epsilon \ll 1 \}$ the Mellin-Barnes integral 
$$ \frac {1}{2 \pi i} \int_{\Pi_\alpha} \prod_{\nu=0}^N \Gamma(q_\nu w)\prod_{k=1}^r \Gamma(1-d_k w) s^{-w} dw,   $$
with $s=(-1)^Qt$ gives us a solution that is the $P_\alpha^{\mu_\alpha-1}$ part of \eqref{eq:Gamma-series}.
This can be seen from the following calculation
$$ \sum_{n \geq 0} Res_{w =-\rho_\alpha -n} \prod_{\nu=0}^N \Gamma(q_\nu w)\prod_{k=1}^r \Gamma(1-d_k w) s^{-w}$$
$$ = \sum_{n \geq 0} \frac{1}{(\mu_\alpha-1)!} (\frac{d}{dw})^{\mu_\alpha-1} ((w +\rho_\alpha + n)^{\mu_\alpha} \prod_{\nu=0}^N\Gamma(q_\nu w)\prod_{k=1}^r \Gamma(1-d_k w) s^{-w})|_{w =-\rho_\alpha -n}$$
$$ = \sum_{n \geq 0} \frac{1}{(\mu_\alpha-1)!} (-\frac{d}{dP})^{\mu_\alpha-1} ((-P)^{\mu_\alpha} \prod_{\nu=0}^N \Gamma(-q_\nu (\rho_\alpha + n+P))\prod_{k=1}^r \Gamma(1+d_k (\rho_\alpha + n+P)) s^{\rho_\alpha +n+P}|_{P=0}$$
$$ = -\sum_{n \geq 0} \frac{(-1)^{nQ}}{(\mu_\alpha-1)!} \sum_{\kappa= 0}^{\mu_\alpha-1}  \binom{\mu_\alpha-1}{\kappa} \left((\frac{d}{dP})^{\mu_\alpha-1-\kappa} \frac{\prod_{k=1}^r \prod_{j=1}^{n d_k}(d_k (\rho_\alpha+P)+j)}{\prod_{\nu=0}^N \prod_{i=1}^{n q_\nu}(q_\nu (\rho_\alpha +P)+i)} ((-1)^{Q}t)^{\rho_\alpha + n+P}\right)|_{P=0}$$
$$\times \left((\frac{d}{dP})^{\kappa} P^{\mu_\alpha} \prod_{\nu=0}^N\Gamma(-q_\nu (\rho_\alpha +P)) \prod_{k=1}^r \Gamma(1+d_k (\rho_\alpha + n+P))\right){\Large \mid}_{P=0}$$

To get a solution with $P_\alpha^{\mu_\alpha-2}$ part of \eqref{eq:Gamma-series} we choose $\nu_1 \in [0,N]$
such that $\rho_\alpha = \frac{a}{q_{\nu_1}}$ for some $a \in [0, q_{\nu_1}-1]$
and calculate
$$ \frac {1}{2 \pi i} \int_{\Pi_\alpha} (-1)^{q_{\nu_1} w}\frac{\prod_{\nu=0, \nu \not = \nu_1}^N \Gamma(q_\nu w)\prod_{k=1}^r \Gamma(1-d_k w) s^{-w}}{\Gamma(1-q_{\nu_1}w)} dw   $$

In this way we increase the number of $\Gamma$-factors in the denominator. The factor $ \Gamma(q_{\nu}w) $ multiplied by a function $\frac{sin \;(\pi q_{\nu}w)}{\pi} (-1)^{q_{\nu}w} $ with period $2 \pi \sqrt{-1}$ gives $\frac{(-1)^{q_\nu w}}{\Gamma(1-q_{\nu}w)}.$ Thus we obtain a $\mu_\alpha$ tuple
of Mellin-Barnes integral solutions to \eqref{eq:hge} that are linear combinations of \eqref{eq:Gamma-series}. To get \eqref{eq:Gamma-series} solutions from the Mellin-Barnes integral solutions we need only to solve a system of linear equations determined by a $\mu_\alpha \times \mu_\alpha$  upper triangle matrix with non-zero diagonal entries.  

As for the statement on the solutions to the irreducible operator $H^{red}$
we shall consider the Mellin-Barnes integrals 

$$ \frac {1}{2 \pi \sqrt{-1} } \int_{\Pi_\alpha} \frac{\prod_{\nu=0}^N \Gamma(q_\nu w)}{ \prod_{k=1}^r \Gamma(d_k w)} s^{-w} dw,   $$
whose poles $w \in \Pi_\alpha$ are at most of order $\mu_\alpha - \nu_\alpha.$
On calculating its residues, we obtain a subspace of solutions to \eqref{eq:hge} of dimension $Q^{red} = \sum_{\alpha=1}^p (\mu_\alpha - \nu_\alpha). $
\end{proof}

\begin{corollary} \label{cr:z0}
There is a basis
$$
 \X(z) = (X_1(z), \dots X_Q(z))
$$
of solutions to \eqref{eq:hge}
such that the monodromy  around $z = 0$ is given by
$$
 \X(z) \to \X(z) \cdot E_0.
$$
\end{corollary}

\begin{proof}
Recall that
$
 \lambda = \left. \prod_{\nu=0}^N q_\nu^{q_\nu}
  \right/ \prod_{k=1}^r d_k^{d_k}
$
and $t = \lambda z$.
The monodromy of the $\Gamma$-series in \eqref{eq:Gamma-series2}
around $t=0$ comes only from
\begin{align}
 e^{P_\alpha \log t} t^{\rho_\alpha}
  &\mapsto e^{P_\alpha \log (\exp(2 \pi \sqrt{-1}) \, t)} (\exp(2 \pi \sqrt{-1}) \, t)^{\rho_\alpha} \\
  &= e^{2 \pi \sqrt{-1} P_\alpha} \cdot \exp(2 \pi \sqrt{-1} \rho_\alpha)
    \cdot e^{P_\alpha \log t} t^{\rho_\alpha},
\end{align}
and other terms are single-valued.
Now Corollary \ref{cr:z0}
follows from the fact that the linear operator
\begin{align*}
 e^{2 \pi \sqrt{-1} P_\alpha} &\cdot \exp(2 \pi \sqrt{-1} \rho_\alpha) \\
  &= \lb 1 + 2 \pi \sqrt{-1} P_\alpha + \cdots
   + \frac{1}{(\mu_\alpha-1)!} \lb 2 \pi \sqrt{-1} P_\alpha \rb^{\mu_\alpha-1} \rb \cdot e_\alpha
\end{align*}
on the $\mu_\alpha$-dimensional vector space
$\bC[P_\alpha]/(P_\alpha^{\mu_\alpha})$
has Jordan normal form
$
 e_\alpha \id_{\mu_\alpha} + J_{\mu_\alpha, -}.
$

Alternatively,
one can also argue as follows:
This can be seen by a generalization of Frobenius method
for a differential equation with multiple local exponents at $s=0.$
The residue
$$
 \Res_{w =-\rho_\alpha } \prod_{\nu=0}^N
  \Gamma(q_\nu w)\prod_{k=1}^r \Gamma(1-d_k w) s^{-w}
$$
admits an asymptotic expansion of the following type
$$
 \sum_{j=1}^{\mu_\alpha}(\log s)^{\mu_\alpha - j} 
 s^{\rho_\alpha} h_j(s)
$$
for holomorphic functions $h_j(s)$,
$j=1,\cdots ,\sigma_\alpha$ at $s=0$.
If we apply the monodromy action
$s \mapsto s e^{2 \pi \sqrt -1}$
to this series $\mu_\alpha$ times repeatedly,
we get a local basis of a rank $\mu_\alpha$ subspace of solutions.
The monodromy action induced by a turn around $s=0$
on this subspace of solutions has the Jordan normal form
$e_\alpha \id_{\mu_\alpha} + J_{\mu_\alpha, -}.$
For each $\alpha$, this subspace is invariant
with respect to the monodromy action within the total solution space,
we get the statement.
\end{proof}

Set $\sigma_0 = 0$ and
$
 \sigma_i = \sum_{\alpha = 1}^i \mu_\alpha
$
for $i = 1, \dots, p$.

\begin{lemma} \label{lm:A}
$X_{\sigma_i}(z)$ is singular at $z = 1$ for any $1 \le i \le p$.
\end{lemma}

\begin{proof}
Assume that $X_{\sigma_i}(z)$ is holomorphic at $z = 1$.
Since $X_{\sigma_i}(z)$ is a solution to \eqref{eq:hge},
its only possible singular points on $\bC$ are $z = 0$ and $1$,
so that $z^{- \rho_i} X_{\sigma_i}(z)$ is in fact an entire function.
Since \eqref{eq:hge} has a regular singularity at infinity,
$X_{\sigma_i}(z)$ has at most polynomial growth at infinity.
This implies that $z^{- \rho_i} X_{\sigma_i}(z)$ is a polynomial,
which cannot be the case since the series \eqref{eq:Gamma-series}
defining $X_{\sigma_i}(z)$ around the origin is infinite.
\end{proof}

\begin{lemma} \label{lm:around-1}
There is a fundamental solution
$\Y(z) = (Y_1(z), \dots, Y_Q(z))$
of \eqref{eq:hge} around $z = 1$
such that $Y_i(z)$ is holomorphic for $i = 1, \dots, Q - 1$.
\end{lemma}

\begin{proof}
We prove the following stronger result;
$Y_i$ has a series expansion
$$
 Y_i = (z - 1)^{i - 1} \sum_{m \ge 0} G_m (z - 1)^m
$$
for $i = 1, \dots, Q - 1$, and
$Y_Q(z)$ has the series expansion
$$
 Y_Q(z)
  = (z - 1)^{\frac{n-1}{2}} \sum_{m \ge 0} G'_m (z - 1)^m
      + \sum_{m \ge 0} G''_m (z - 1)^m
$$
when $n$ is even, and
$$
 Y_Q(z)
  = (z - 1)^{\frac{n-1}{2}} \log (z - 1)
     \left(
      \sum_{m \ge 0} G'_m (z - 1)^m
     \right)
     + \sum_{m \ge 0} G''_m (z - 1)^m
$$
when $n$ is odd.
Since $Q - 2$ is the largest exponent,
one can find a series solution
$$
 Y_{Q - 1} = (z - 1)^{Q - 2} \sum_{m \ge 0} G_m (z - 1)^m
$$
to \eqref{eq:hge}.

Next we remove the common factor $\theta$ in two terms of \eqref{eq:hge} from the left
to obtain a differential equation;
the factor $q_0 \theta_t$
with $\nu = a = 0$ in the first term is replaced with $q_0$,
and the factor $(d_1 \theta_t + d_1)$
with $k = 1$ and $b = d_1$
in the second term is replaced with $d_1$
(note that $t (d_1 \theta_t + d_1) = d_1 \theta_t t$).
One can see
(cf. e.g. \cite[(2.8)]{Beukers-Heckman})
that the set of local exponents of the resulting equation
at $z = 1$ is given by
$$
 \lc 0, 1, \dots, Q - 3, \frac{n - 1}{2} \rc.
$$
Now $Q - 3$ is the largest exponent, and
one can find a series solution
$$
 Y_{Q - 2} = (z - 1)^{Q - 3} \sum_{m \ge 0} G_m (z - 1)^m
$$
to this equation.

One can continue this process
until the differential equation becomes irreducible of rank $Q^{\reduced}$
whose  set of exponents is given by
$
 \lc 0, 1, \dots, Q^{\reduced}-2, \frac{n-1}{2} \rc.
$
This irreducible differential equation describes the period
of the Calabi-Yau manifold obtained by compactifying $Y_t$
(cf. \cite[Theorem 1.1]{Corti-Golyshev}).
This Calabi-Yau manifold has an ordinary double point at $z = 1$,
and the period integral along the vanishing cycle
gives the singular solution $Y_Q(z)$,
while integrals against classes
orthogonal to the vanishing cycle give holomorphic solutions
$Y_1(z), \dots, Y_{Q^{\reduced}-1}(z)$.
\end{proof}

Lemma \ref{lm:A} and Lemma \ref{lm:around-1} imply the following:

\begin{lemma} \label{lm:A'}
One can choose a fundamental solution
$\Y(z) = (Y_1(z), \dots, Y_Q(z))$
around $z = 1$ so that
the connection matrix
\begin{equation}
 \X(z) = \Y(z) \cdot L_1
\end{equation}
is given by
\begin{equation}
 L_1 =
 \begin{pmatrix}
   1 & 0 & \cdots & 0 & 0 \\
   0 & 1 & \cdots & 0 & 0 \\
   \vdots & \vdots & \ddots & \vdots & \vdots \\
   0 & 0 & \cdots & 1 & 0 \\
   c_1 & c_2 & \cdots & c_{Q - 1} & 1 \\
 \end{pmatrix}
\end{equation}
where
$
 c_{\sigma_i} \ne 0
$
for any $i = 1, \dots, p$.
\end{lemma}

When $n$ is odd,
the monodromy of $Y_Q$ around $s = 1$ is given by
$$
 Y_Q(z)
  \to Y_Q(z) + 2 \pi \I (z - 1)^{(n - 1)/2}
                  \sum_{m = 0}^\infty G'_m (z - 1)^m.
$$
The second term is holomorphic at $z = 1$ and 
can be expressed as a linear combination
of $Y_1(z), \ldots, Y_{Q-1}(z)$.
Hence the monodromy around $z = 1$ is given by
$$
 \Y(z) \to \Y(z) \cdot E_1
$$
where
$$
 E_1 =
 \begin{pmatrix}
  1 & 0 & \cdots & 0 & c'_1 \\
  0 & 1 & \cdots & 0 & c'_2 \\
  \vdots & \vdots & \ddots & \vdots & \vdots \\
  0 & 0 & \cdots & 1 & c'_{Q-1} \\
  0 & 0 & \cdots & 0 & 1 \\
 \end{pmatrix}
$$
When $n$ is even,
$$
 Y_Q(z)
  \to - Y_Q(z) + 2 \sum_{m=0}^\infty G'_m (z - 1)^m,
$$
so that the monodromy around $z = 1$ is given by
$$
 \Y(z) \to \Y(z) \cdot E_1.
$$
where
$$
 E_1 =
 \begin{pmatrix}
  1 & 0 & \cdots & 0 & c'_1 \\
  0 & 1 & \cdots & 0 & c'_2 \\
  \vdots & \vdots & \ddots & \vdots & \vdots \\
  0 & 0 & \cdots & 1 & c'_{Q-1} \\
  0 & 0 & \cdots & 0 & - 1 \\
 \end{pmatrix}.
$$
Note that the monodromy of $\Y(z)$ around $z = 0$
is given by
\begin{align*}
 \Y(&z) = \X(z) \cdot L_1^{-1} \\
  &\to \X(z) \cdot E_0 \cdot L_1^{-1}
    = \Y(z) \cdot L_1 \cdot E_0 \cdot L_1^{-1}.
\end{align*}
By a straightforward calculation,
we have the following:
\begin{proposition}
The monodromy matrices $h_0$, $h_1$ and $h_\infty$
around $z = 0$, $1$ and $\infty$
with respect to the basis $\Y(z)$
of solutions of \eqref{eq:hge} are given by
$$
 h_0
  = E_0
   + \begin{pmatrix}
      0 & 0 & \cdots & 0 \\
      \vdots & \vdots & \ddots & \vdots \\
      0 & 0 & \cdots & 0 \\
      \gamma_1 & \gamma_2 & \cdots & \gamma_Q
     \end{pmatrix},
$$
$$
 h_1
  = \begin{pmatrix}
      1 & 0 & \cdots & 0 & g_1 \\
      0 & 1 & \cdots & 0 & g_2 \\
      \vdots & \vdots & \ddots & \vdots & \vdots \\
      0 & 0 & \cdots & 1 & g_{Q-1} \\
      0 & 0 & \cdots & 0 & (-1)^{n-1}
     \end{pmatrix},
$$
$$
 h_\infty^{-1}
  = h_0
   + \begin{pmatrix}
      0 & 0 & \cdots & 0 & \delta_1 \\
      0 & 0 & \cdots & 0 & \delta_2 \\
      \vdots & \vdots & \ddots & \vdots & \vdots \\
      0 & 0 & \cdots & 0 & \delta_Q
     \end{pmatrix}.
$$
% \begin{align*}
%  h_0
%   &= E_0
%    + \begin{pmatrix}
%       0 & 0 & \cdots & 0 \\
%       \vdots & \vdots & \ddots & \vdots \\
%       0 & 0 & \cdots & 0 \\
%       \gamma_1 & \gamma_2 & \cdots & \gamma_Q
%      \end{pmatrix} \\
%  h_1
%   &= \begin{pmatrix}
%       1 & 0 & \cdots & 0 & g_1 \\
%       0 & 1 & \cdots & 0 & g_2 \\
%       \vdots & \vdots & \ddots & \vdots & \vdots \\
%       0 & 0 & \cdots & 1 & g_{Q-1} \\
%       0 & 0 & \cdots & 0 & (-1)^{n-1}
%      \end{pmatrix} \\
%  h_\infty^{-1}
%   &= h_0
%    + \begin{pmatrix}
%       0 & 0 & \cdots & 0 & \delta_1 \\
%       0 & 0 & \cdots & 0 & \delta_2 \\
%       \vdots & \vdots & \ddots & \vdots & \vdots \\
%       0 & 0 & \cdots & 0 & \delta_Q
%      \end{pmatrix}
% \end{align*}
\end{proposition}

\begin{lemma} \label{lm:B}
Let $v = (v_1, \dots, v_Q)^T$ be a column vector
and define a $Q \times Q$ matrix by
$$
 T = (v, h_0 \cdot v, \dots, h_0^{Q-1} \cdot v).
$$
Then one has
$$
 \det T
  = \pm \prod_{1 \le \alpha < \beta \le p}
         (e_\alpha - e_\beta)^{\mu_\alpha \cdot \mu_\beta}
     \cdot \prod_{\alpha=1}^p (v_{\sigma_{\alpha-1} + 1})^{\mu_\alpha}.
$$
\end{lemma}

\begin{proof}
Let $T(\alpha, j) \in \SL_Q(\bC)$ be the block diagonal matrix
defined by
$$
 T(\alpha, j) =
 \begin{pmatrix}
   \id_{Q-j-1} & 0 \\
   0 & \id_{j+1} - e_{\alpha} \cdot J_{j+1, +}
 \end{pmatrix}.
$$
Then
\begin{align*}
 T \cdot T(&1, Q - 1) \cdot T(1, Q - 2) \cdot \cdots \cdot T(1, Q - \mu_1) \\
   & \cdot T(2, Q - \mu_1 - 1) \cdot \cdots \cdot T(2, Q - \mu_1 - \mu_2) \\
   & \cdot T(p, Q - \sigma_{p-1} - 1) \cdot \cdots \cdot T(p, 1) \\
\end{align*}
is a lower-triangular matrix
whose $i$-th diagonal component
for $\sigma_{\alpha - 1} < i \le \sigma_\alpha$ is given by
$$
 \prod_{\beta < \alpha}
  (e_\alpha - e_\beta)^{\mu_\beta}
 \cdot v_{\sigma_{\tiny \alpha-1}+1}.
$$
\end{proof}

\begin{corollary}
$v = (v_1, \dots, v_Q)^T$ is a cyclic vector with respect to $h_0$
if and only if the condition
\begin{equation} \label{eq:cyclic}
 \prod_{\alpha=1}^p v_{{\sigma_{\alpha-1}} + 1} \ne 0  
\end{equation}
is satisfied.
\end{corollary}

\begin{lemma}
If $v \in \bC^Q$ satisfies
\begin{equation} \label{eq:i}
 h_\infty^{-i} \cdot v = h_0^i v,
  \qquad i = 1, 2, \dots, Q-1,
\end{equation}
then \eqref{eq:cyclic} holds.
\end{lemma}

\begin{proof}
Since the kernel of $h_\infty^{-1} - h_0$ is
the orthogonal complement of the last coordinate vector
$
 \be_Q = (0, \dots, 0, 1) \in \bC^Q,
$
the equations \eqref{eq:i} for $v = (\bv, 0)$
where $\bv=(v_1, \cdots, v_{Q-1})$
can be rewritten as
$$
 \Sigma \cdot \bv = 0
$$
where $\Sigma$ is a $(Q - 1) \times (Q-1)$ matrix
whose $j$-th row vector is
the first $Q - 1$ components of the last row vector of $h_0^j$.
Define a block diagonal $(Q - 1) \times (Q-1)$ matrix by
$$
 S(\alpha, j)
  =
  \begin{pmatrix}
   \id_{Q-j-2} & 0 \\
   0 & S'
  \end{pmatrix}
$$
where $S' \in \SL_{j+1}(\bC)$ is given by
$$
 S'
  =
  \begin{pmatrix}
   1 & 0 & \cdots & 0 & 0 \\
   -e_\alpha & 1 & \cdots & 0 & 0 \\
   \vdots & \vdots &\ddots & \vdots & \vdots \\
   0 & 0 & \cdots & 1 & 0 \\
   0 & 0 & \cdots & - e_\alpha & 1
  \end{pmatrix}.
$$
Then the components of the matrix
\begin{align*}
 \Sigmatilde
  &= S(1, 1) \cdots S(1, \mu_1 - 1)
     \cdot S(2, \mu_1) \cdots S(2, \sigma_2 - 1)
     \cdot S(3, \sigma_2) \cdots S(3, \sigma_3 - 1) \\
  & \qquad \cdots S(p, \sigma_{p-1}) \cdots S(p, \sigma_p - 2)
    \cdot \Sigma
\end{align*}
are zero below the anti-diagonal
(i.e., $\Sigmatilde_{i j} = 0$ if $i + j > Q$)
and the $i$-th anti-diagonal component
$
\Sigmatilde_{i, Q - i - 1}
$
for $\sigma_{\alpha - 1} < i \le \sigma_\alpha$ is given by
$$
 \prod_{\beta > \alpha}
  (e_\alpha - e_\beta)^{\mu_\beta} c_{\sigma_{\alpha}}.
$$
The $(Q - 1)$-st equation
$$
 \text{(const)} \cdot v_1
  + \prod_{\beta>1} (e_1 - e_\beta)^{\mu_\beta} c_{\mu_1} v_2
  = 0
$$
together with Lemma \ref{lm:A'} implies that $v_2 = 0$
if $v_1 = 0$.
By repeating this type of argument,
one shows that $v_1 = 0$ implies $\bv = 0$.
Moreover, one can run the same argument
by interchanging the role of $(v_1, e_1, c_{\mu_1})$
with $(v_{\sigma_{\alpha-1}+1}, e_\alpha, c_{\sigma_\alpha})$
to show that $v_{\sigma_{\alpha-1}+1} = 0$ implies $\bv = 0$.
Hence a non-trivial solution to \eqref{eq:i}
must satisfy \eqref{eq:cyclic}.
\end{proof}

This concludes the proof of Theorem \ref{th:monodromy-group}.

\section{Invariants of the hypergeometric group} \label{sc:ihg}

We prove the following
in this section:

\begin{proposition} \label{prop:ihg}
Let $\bq = (q_0, \dots, q_N)$ and $\bd = (d_1, \dots, d_r)$
be sequences of positive integers
such that $Q := \sum_{i=0}^N q_i = \sum_{k=1}^r d_r$.
Then the space of $Q \times Q$ matrices
invariant under the action
$$
 \Hqd \ni h
  : X \mapsto h \cdot X \cdot h^T
$$
is at most one-dimensional.
\end{proposition}

\begin{proof}

Let $X$ be a $Q \times Q$ matrix
invariant under the hypergeometric group $\Hqd$,
so that
\begin{equation} \label{eq:h-invariance}
 h \cdot X \cdot h^T = X
\end{equation}
for any
$
 h \in \Hqd.
$
%Let $\{ e_i \}_{i=0}^N$ be the standard basis of $\bC^{N+1}$.
%Since
%$$
% h_0^T =
%  \begin{pmatrix}
%   0 & 1 & 0 & \cdots & 0 \\
%   0 & 0 & 1 & \cdots & 0 \\
%   \vdots & \vdots & \vdots & \ddots & \vdots \\
%   0 & 0 & 0 & \cdots & 1 \\
%   -A_{N+1} & - A_{N} & - A_{N-1} & \cdots & - A_1
%  \end{pmatrix}
%$$
%and $A_{N+1} = \pm 1$,
%we have
%\begin{align*}
% e_N &= - A_{N+1} h_0^T e_1, \\
% e_{N_1} &= h_0^T e_N + A_1 e_N, \\
% e_{N_2} &= h_0^T e_{N-1} + A_2 e_N, \\
% \dots, \\
% e_2 &= h_0^T e_3 + A_N e_N.
%\end{align*}
Let $e_1 = (1, 0, \dots, 0)^T$ be the first coordinate vector.
Since $\{ (h_\infty^T)^i e_1 \}_{i=0}^Q$ spans $\bC^{Q}$,
$X_{ij}$ is determined by the $\Hqd$-invariance
once we know $X_{i1}$ for $i=1, \dots, Q$.
Put
\begin{equation} \label{eq:h1}
 h_1 = h_0^{-1} \cdot h_\infty^{-1}
   = \begin{pmatrix}
      (-1)^{r} B_Q & 0 & \cdots & 0 & 0 \\
      (-1)^{r} (B_{Q-1} - A_{Q-1}) & 1 & \cdots & 0 & 0 \\
      \vdots & \vdots & \ddots & \vdots & \vdots \\
      (-1)^{r} (B_2 - A_2) & 0 & \cdots & 1 & 0 \\
      (-1)^{r} (B_1 - A_1) & 0 & \cdots & 0 & 1 \\
     \end{pmatrix}
  \in \Hqd
\end{equation}
%be the matrix satisfying
%$
% h_0 \cdot h_1 \cdot h_\infty = 1
%$
and consider \eqref{eq:h-invariance} for $h = h_1$.
%where we have used $A_Q = (-1)^r$ and
%\begin{equation} \label{eq:hinfty-inverse}
% h_\infty^{-1}
%   = \begin{pmatrix}
%      (-1)^{r+1} A_{Q-1} & 1 & 0 & \cdots & 0 \\
%      (-1)^{r+1} A_{Q-2} & 0 & 1 & \cdots & 0 \\
%      \vdots & \vdots & \vdots & \ddots & \vdots \\
%      (-1)^{r+1} A_1 & 0 & 0 & \cdots & 1 \\
%      (-1)^{r+1} & 0 & 0 & \dots & 0
%     \end{pmatrix}.
%\end{equation}
%Note that
%\begin{align*}
% (h_1)_{ij}
%  &= (h_\infty \cdot h_0^{-1})_{ij} \\
%  &=
%\begin{cases}
% (-1)^{N-r} & i = j = 1, \\
% (-1)^r (B_{Q-i+1} - A_{Q-i+1})& i \ne 1 \text{ and } j - 1, \\
% \delta_{ij} & i \ne 1 \text{ and } j \ne 1
%\end{cases}
%\end{align*}
%and consider the relation
%\begin{equation} \label{eq:h1_invariance}
% X = h_1 \cdot X \cdot h_1^T. 
%\end{equation}
Since
\begin{align*}
 (h_1 \cdot X \cdot h_1^T)_{i1}
   &= \sum_{k,l=1}^Q (h_1)_{ik} X_{kl} (h_1)_{1l} \\
   &= \sum_{k,l=1}^Q (h_1)_{ik} X_{kl} (-1)^{r+N+1} \delta_{1l} \\
   &= \sum_{k=1}^Q (-1)^{N+r+1} (h_1)_{ik} X_{k1},
%   &= (-1)^{N+r+1} ((h_1)_{i1} X_{11} + X_{i1}),
\end{align*}
the first column of \eqref{eq:h-invariance} reduces to
\begin{equation} \label{eq:h1-invariance}
 (-1)^{N+r+1}((h_1)_{i1} X_{11} + X_{i1}) = X_{i1}
\end{equation}
for $2 \le i \le Q$.
If $n = N - r$ is even,
then \eqref{eq:h1-invariance} implies
$$
 X_{i1} = - \frac{1}{2} (h_1)_{i1} X_{11},
$$
so that the space of $\Hqd$-invariants
is at most one-dimensional.
If $N + r$ is odd,
then \eqref{eq:h1-invariance} gives
$
 X_{11} = 0.
$
Fix $j \ne 1$ such that
$(h_1)_{j1} = (-1)^r (B_{Q-j+1} - A_{Q-j+1}) \ne 0$.
Since
\begin{align*}
 (h_1 \cdot X \cdot h_1^T)_{ij}
   &= \sum_{k,l=1}^Q (h_1)_{ik} X_{kl} (h_1)_{jl} \\ 
   &= \sum_{k=1}^Q (h_1)_{ik} ( X_{k1} (h_1)_{j1} + X_{kj} (h_1)_{jj}) \\ 
   &= \sum_{k=1}^Q (h_1)_{ik} ( X_{k1} (h_1)_{j1} + X_{kj} ) \\ 
   &= (h_1)_{i1} ( X_{11} (h_1)_{j1} + X_{1j} )
        + ( X_{i1} (h_1)_{j1} + X_{ij} ) \\ 
   &= (h_1)_{i1} X_{1j}
        + X_{i1} (h_1)_{j1} + X_{ij},
\end{align*}
the $j$-th column of \eqref{eq:h-invariance} gives
$$
 (h_1)_{i1} X_{1j} + (h_1)_{j1} X_{i1} = 0
$$
for $2 \le i \le Q$.
Since $(h_1)_{j1} \ne 0$,
one obtains
$$
 X_{i1} = - \frac{(h_1)_{1i}}{(h_1)_{j1}} X_{1j}
$$
for $2 \le i \le Q$,
so that the space of $H$-invariants
is at most one-dimensional
also in this case.
\end{proof}

\section{Coherent sheaves on Calabi-Yau complete intersections
in weighted projective spaces} \label{sc:CYci}

We prove the $\Hqd$-invariance
of the Gram matrix in Theorem \ref{th:invariant}
in this section.
The proof is closely related
to the discussion
of Golyshev \cite[\S 1]{Golyshev_RRV},
although the use of the right dual collection
$(\scFtilde_i)_{i=1}^Q$
seems to be new.

% We also show that there is a map
% from the lattice of vanishing cycles
% to the Grothendieck group of the Calabi-Yau complete intersection
% such that the $\gamma_i$ and $\epsilon_i$
% given in section \ref{sc:ihg} and \ref{sc:imat}
% correspond to a natural generators of the derived category.

Let $Y$ be a smooth complete intersection
of degree $(d_1, \ldots, d_r)$
in the weighted projective space
$\bP = \bP(q_0, \dots, q_N)$.
We use the Koszul resolution
\begin{align}
\begin{split} \label{eq:KoszulY}
 0 & \lto \scO_\bP(-d_1 - \dots - d_r)
     \lto \bigoplus_{i=1}^r
      \scO_\bP( - d_1 - \dots - \widehat{d_i} - \dots - d_r ) \\
   & \quad \qquad \lto \cdots
     \lto \bigoplus_{1 \le i < j \le r}
      \scO_\bP(-d_i - d_j)
     \lto \bigoplus_{i=1}^r \scO_\bP(-d_i)
     \lto \scO_\bP
     \lto \iota_* \scO_Y
     \lto 0
\end{split}
\end{align}
of the structure sheaf $\scO_Y$ of $Y$
to compute the derived restriction
$$
 \bL \iota^*(-) := (-) \Lotimes_{\iota^{-1}\scO_\bP} \scO_Y
  : D^b \coh \bP \to D^b \coh Y,
$$
where $\iota : Y \hookrightarrow \bP$ is the inclusion.

Let $(\scEtilde_i)_{i=1}^Q$ be the full strong exceptional collection
on $D^b \coh \bP$
given as
$$
 (\scEtilde_1, \dots, \scEtilde_Q) = (\scO, \dots, \scO(Q-1)),
$$
and $(\scFtilde_1, \dots, \scFtilde_Q)$ be
its right dual exceptional collection
characterized by the condition
$$
 \Ext^k(\scEtilde_{Q-i+1}, \scFtilde_j) = 
  \begin{cases}
   \bC & i=j, \text{ and } k=0, \\
   0 & \text{otherwise}.
  \end{cases}
$$
Note that
$
 \scFtilde_1 = \scO_\bP(-1)[N]
$
and
$
 \scFtilde_Q = \scE_1 = \scO_\bP.
$
%One has
%$$
% (\scFtilde_0, \dots, \scFtilde_N)
%  = (\Omega^N(N)[N], \Omega^{N-1}(N-1)[N-1], \dots, \scO)
%$$
%when $q_1 = \dots = q_N = 1$.
The Euler form on the Grothendieck group $K(\bP)$
defined by \eqref{eq:Euler}
is neither symmetric nor anti-symmetric,
whereas that on $K(Y)$ is either symmetric or anti-symmetric
depending on the dimension of $Y$.
The bases $\{ [\scEtilde_i] \}_{i=1}^Q$
and $\{ [\scFtilde_i] \}_{i=1}^Q$ of $K(\bP)$
are dual to each other in the sense that
\begin{align} \label{eq:Euler_form}
 \chi(\scEtilde_{Q-i+1}, \scFtilde_j) = \delta_{ij}.
\end{align}
We will write the derived restrictions
of $\scEtilde_i$ and $\scFtilde_i$
to $Y$ as
$\scEbar_i$ and $\scFbar_i$ respectively.
Unlike $\{ [\scEtilde_i] \}_{i=1}^Q$
and $\{ [\scFtilde_i] \}_{i=1}^Q$,
$\{ [\scEbar_i] \}_{i=1}^Q$ and
$\{ [\scFbar_i] \}_{i=1}^Q$ are not bases of $K(Y)$,
and their images in the numerical Grothendieck group
are linearly dependent.
Put
\begin{align*}
\Xbar_{ij} &= \chi([\scFbar_i], [\scFbar_j])
% \Xtilde_{ij} &= \chi([\scFtilde_i], [\scFtilde_j]), &
%  \quad \Xbar_{ij} &= \chi([\scFbar_i], [\scFbar_j]), \\
% \Ytilde_{ij} &= \chi([\scEtilde_i], [\scEtilde_j]), &
%  \quad \Ybar_{ij} &= \chi([\scEbar_i], [\scEbar_j]),
\end{align*}
and let $(a_{ij})_{i,j=1}^Q$ be
the transformation matrix
between two bases $\{ [\scEtilde_i] \}_{i=1}^Q$
and $\{ [\scFtilde_i] \}_{i=1}^Q$
so that
$$
 [\scFtilde_i] = \sum_{j=1}^Q [\scEtilde_j] a_{ji}.
$$
%
%We will use the following notations in this section;
%$\Xtilde = (\Xtilde_{ij})_{i,j=1}^Q$
%$\Xbar = (\Xbar_{ij})_{i,j=1}^Q$ are
%Gram matrices
%$$
% \Xtilde_{ij} = \chi([\scFtilde_i], [\scFtilde_j]), \quad
% \Xbar_{ij} = \chi([\scFbar_i], [\scFbar_j])
%$$
%with respect to the elements
%$\{ [\scFtilde_i] \}_{i=1}^Q$
%and $\{ [\scFbar_i] \}_{i=1}^Q$ respectively,
%$\Ytilde = (\Ytilde_{ij})_{i,j=1}^Q$
%$\Ybar = (\Ybar_{ij})_{i,j=1}^Q$ are
%Gram matrices
%$$
% \Ytilde_{ij} = \chi([\scEtilde_i], [\scEtilde_j]), \quad
% \Ybar_{ij} = \chi([\scEbar_i], [\scEbar_j])
%$$
%of the bases
%$\{ [\scEtilde_i] \}_{i=1}^Q$
%and $\{ [\scEbar_i] \}_{i=1}^Q$ respectively,
%and $(a_{ij})_{i,j=1}^Q$ is the transformation matrix
%between two basis $\{ \scEtilde_i \}_{i=1}^Q$
%and $\{ \scFtilde_i \}_{i=1}^Q$,
%so that
%$$
% [\scFtilde]_i = \sum_{j=1}^Q [\scEtilde_j] a_{ji}.
%$$
%

We prove the following in this section:

\begin{proposition} \label{prop:Xbar-ihg}
$\Xbar$ is an invariant of the hypergeometric group
$\Hqd$.
\end{proposition}

We divide the proof into three steps.

\begin{lemma} \label{lm:autoeq_isometry}
Let $\Phi$ be an autoequivalence of $D^b \coh Y$
such that its action on $\{ [\scFbar_i] \}_{i=1}^Q$
is given by
$$
 [\scFbar_i] \mapsto \sum_{j=1}^Q h_{ij} [\scFbar_j].
$$
Then $\Xbar$ is invariant under the action of
$h = (h_{ij})_{i,j=1}^Q$;
$$
 \Xbar = h \cdot \Xbar \cdot h^T.
$$
\end{lemma}

\begin{proof}
Since an autoequivalence $\Phi$ induces an isometry of $K(Y)$,
one has
\begin{align*}
 \Xbar_{i j}
  &= \chi([\scFbar_i], [\scFbar_j]) \\
  &= \chi([\Phi(\scFbar_i)], [\Phi(\scFbar_j)]) \\
  &= \sum_{k, l=1}^Q \h_{i k} \chi([\scFbar_k], [\scFbar_l]) \h_{j l} \\
  &= \sum_{k, l=1}^Q \h_{i k} \Xbar_{k l} \h_{j l}
\end{align*}
for any $1 \le i,j \le Q$.
\end{proof}

\begin{remark}
Since $\{ [\scFbar_i] \}_{i=1}^Q$ are not linearly independent,
the choice of $h$ in Lemma \ref{lm:autoeq_isometry}
is not unique.
\end{remark}

\begin{lemma} \label{lm:hinfty}
The action of the autoequivalence of $D^b \coh Y$
defined by the tensor product with $\scO_Y(-1)$
on $\{ \scFbar_i \}_{i=1}^Q$
is given by $h_0$;
$$
 [\scFbar_i \otimes \scO_Y(-1)]
   = \sum_{j=1}^Q (h_0)_{ij} [\scFbar_j].
$$
\end{lemma}

\begin{proof}
Since tensor product with $\scO(-1)$
commutes with restriction,
it suffices to show
$$
 [\scFtilde_i \otimes \scO_\bP(-1)]
   = \sum_{j=1}^Q (h_0)_{ij} [\scFtilde_j].
$$
Since $\{ [\scEtilde_{Q-i+1}] \}_{i=1}^Q$ and
$\{ [\scFtilde_i] \}_{i=1}^Q$ are dual bases,
this is equivalent to
\begin{equation} \label{eq:h0-action}
 [\scEtilde_{Q-i+1} \otimes \scO(-1)]
   = \sum_{j=1}^Q [\scEtilde_{Q-j+1}] (h_0^{-1})_{ji}.
\end{equation}
Recall from \eqref{eq:h0} that
$$
 (h_0^{-1})_{ji}
  = \delta_{j, i+1} - \delta_{i, Q} B_{Q-j+1}.
$$
Since $\scEtilde_i = \scO(i-1)$,
Equation \eqref{eq:h0-action} for $i \ne Q$ gives
$$
 \scO_\bP(Q-i) \otimes \scO_\bP(-1) = \scO_\bP(Q-i-1),
$$
which is obvious.
Equation \eqref{eq:h0-action} for $i = Q$ gives
$$
 [\scO_\bP(-1)] + \sum_{j=1}^Q B_{Q-j+1} [\scO_\bP(Q-j)] = 0,
$$
which is $\scO_\bP(-1)$ times the relation
$$
 [\scO_\bP] + B_1 [\scO_\bP(1)]
  + \cdots + B_{Q-1} [\scO_\bP(Q-1)] + B_Q [\scO_\bP(Q)] = 0
$$
coming from the exact sequence
%\begin{align*}
% 0  \to \scO_\bP \lb - \sum_{i=0}^N q_i \rb
%%     \to \bigoplus_{i=0}^N
%%      \scO_\bP \lb - \sum_{j \ne i} q_j \rb
%    \to \dots
%     \to \bigoplus_{0 \le i < j \le N}
%      \scO_\bP(-q_i - q_j)
%     \to \bigoplus_{i=0}^N \scO_\bP(-q_i)
%     \to \scO_\bP
%     \to 0
%\end{align*}
\begin{align} \label{eq:KoszulP}
 0 
 \to \scO_\bP
 \to \bigoplus_{i=0}^N \scO_\bP( q_i )
 \to \bigoplus_{0 \le i < j \le N} \scO_\bP(q_i + q_j)
 \to \cdots
 \to \scO_\bP \lb \sum_{i=0}^N q_i \rb
 \to 0
\end{align}
obtained by sheafifying the Koszul resolution
\begin{align*}
 0 \to \Lambda^N V \otimes \Sym^* V^*
   \to &\cdots \to \Lambda^2 V \otimes \Sym^* V^* \\
   &\to V \otimes \Sym^* V^* \to \Sym^* V^* \to \bC \to 0,
\end{align*}
where $V$ is a graded vector space such that
$\bP = \Proj(\Sym^* V^*)$.
\end{proof}

\begin{lemma} \label{lm:hone}
The action of the autoequivalence of $D^b \coh Y$
given by the dual spherical twist
$
 T_{\scFbar_1}^{\vee}
$
along $\scFbar_1$
is given on $\{ \scFbar_i \}_{i=1}^Q$ by $h_1$;
$$
 [T_{\scFbar_1}^{\vee}(\scFbar_i)]
   = \sum_{i=1}^Q (h_1)_{ij} [\scFbar_j].
$$
\end{lemma}

\begin{proof}
Recall that for a spherical object $\scE$
and an object $\scF$,
the dual spherical twist $T_\scE^{\vee} \scF$
of $\scF$ along $\scE$ is defined as
the mapping cone
$$
 T_\scE^{\vee} \scF
  = \{ \scF \to \hom(\scF, \scE)^{\vee} \otimes \scF \}
$$
of the dual evaluation map.
Since the induced action
%of the dual twist functor $T_{\scE}^{\vee}$
on the Grothendieck group
is given by the reflection
\begin{align*}
 [T_{\scE}^{\vee}(\scF)]
   &= [\scF] - \chi(\scF, \scE) [\scE],
\end{align*}
it suffices to show that
\begin{equation} \label{eq:hX}
 (h_1)_{ij} = \delta_{ij} - \Xbar_{i1} \delta_{j1}.
\end{equation}
Recall from \eqref{eq:h1} that
$$
 (h_1)_{ij} =
\begin{cases}
 (-1)^r B_Q & i = j = 1, \\
 \delta_{ij} & j \ne 1, \\
 (-1)^r(B_{Q-i+1} - A_{Q-i+1}) & i \ne 1 \text{ and } j = 1.
\end{cases}
$$
Equation \eqref{eq:hX} for $j \ne 1$ is obvious,
and that for $i = j = 1$ follows from
$$
 (-1)^r B_Q = (-1)^{r + N + 1}
  = (-1)^{N-r+1}
  = (-1)^{n+1}
$$
and
$$
 \Xbar_{11} =
\begin{cases}
 0 & n \text{ is odd}, \\
 2 & n \text{ is even}.
\end{cases}
$$
To prove \eqref{eq:hX} for $i \ne 1$ and $j = 1$,
one can use
\begin{align*}
 \chi(\scFtilde_i(1))
  = \sum_{j=1}^Q \chi((h_0^{-1})_{ij} \scFtilde_j)
  = \sum_{j=1}^Q (h_0^{-1})_{ij} \chi(\scEtilde_1, \scFtilde_j)
  = (h_{\infty}^{-1})_{iQ}
  = - B_{Q-i+1}.
\end{align*}
and
$$
 \chi(\scFtilde_i(j))
  = \chi(\scO(-j), \scFtilde_i)
  = \chi(\scEtilde_{-j+1}, \scFtilde_i)
  = \delta_{Q+j,i}
$$
for $-Q+1 \le j \le 0$
to show
\begin{align*}
 (-1)^r \Xbar_{i 1}
  &= (-1)^r \chi(\scFbar_i, \scFbar_1) \\
  &= (-1)^{N} \chi(\scFbar_1, \scFbar_i) \\
  &= (-1)^{N} \chi(\scO_Y(-1)[N], \scFbar_i) \\
  &= \chi(\scO_Y(-1), \scFbar_i) \\
  &= \chi(\scFbar_i(1)) \\
  &= \chi(\scFtilde_i(1))
      - \sum_{k=1}^r \chi(\scFtilde_i(1 - d_k))
       + \sum_{1 \le k < l \le r} \chi(\scFtilde_i(1 - d_k - d_l)) \\
  & \qquad - \dots +(-1)^r \chi(\scFtilde_i(1 - d_1 - \dots - d_r)) \\
  &= - B_{Q-i+1}
      - \sum_{k=1}^r \delta_{Q-d_k+1, i}
      + \sum_{1 \le k < l \le r} \delta_{Q - d_k - d_l + 1, i}
      - \cdots + (-1)^r \delta_{Q - d_1 - \dots - d_r + 1, i} \\
  &= - B_{Q-i+1}
      - \sum_{k=1}^r \delta_{Q-i+1, d_k}
      + \sum_{1 \le k < l \le r} \delta_{Q - i + 1, d_k + d_l}
      - \cdots + (-1)^r \delta_{Q - i + 1, d_1 + \dots + d_r} \\
  &= - B_{Q-i+1} + A_{Q-i+1},
\end{align*}
where we have used \eqref{eq:a}
in the last equality.
\end{proof}

\section{Mirror manifolds and Stokes matrices}
 \label{sc:stokes}

In this section,
we discuss the relation between the Gram matrix
in Theorem \ref{th:invariant}
in the case when $Y$ is a hypersurface
and the Stokes matrix for the quantum cohomology
of the weighted projective space.
By \cite[Corollary 1.8]{Coates-Corti-Lee-Tseng},
the quantum differential equation
for the small $J$-function of $\bP$
is given by
\begin{equation} \label{eq:qde}
 \prod_{i=0}^n \prod_{k=0}^{q_i-1}
  \lb q_i z \frac{\partial}{\partial t_1} - k z \rb
   J_\bP = e^{t_1} J_\bP,
\end{equation}
where $t_1$ is the flat coordinate
associated with the positive generator of
$H^2(\bP; \bZ) \subset \Horb^*(\bP; \bC)$
and $z$ is the quantization parameter.
It follows that the stationary-phase integrals
\begin{align} \label{eq:spint}
 J_i(t_1; z)
  = \int_{\Gamma_i} e^{f / z} \Omega
\end{align}
span the identity component
of the space of flat sections
of the first structure connection,
where
$f$ is the function
$
 f(x) = \sum_{i=0}^N q_i x_i
$
on
$$
 \bT = \{ (x_0, \ldots, x_N) \in (\bCx)^{N+1}
  \mid x_0^{q_0} \cdots x_N^{q_N} = e^{t_1} \},
$$
$\Omega$ is the holomorphic volume form
$
 \Omega = d x_0 \wedge \cdots \wedge d x_N
  / d (x_0^{q_0} \cdots x_N^{q_N})
$
on $\bT$, and
$\{ \Gamma_i \}_{i=1}^Q$ is a basis
of flat sections of the local system
whose fiber is the relative homology group
$
 H_N(\bT, \ \Re(f/z) \ll 0; \bZ).
$

The function $f$ has $Q$ critical points
$$
% (x_0, \ldots, x_{N})
 p_i
  = e^{(t_1 - 2 i \pi \sqrt{-1}) / Q} \cdot
     (1, \ldots, 1),
  \qquad
   i = 1, \ldots, Q
$$
with critical values
$$
 f(p_i) = Q e^{(t_1 - 2 i \pi \sqrt{-1}) / Q},
$$
where the minus sign comes from the clockwise order
on the distinguished set $(c_i)_{i=1}^Q$ of vanishing paths,
which we choose as straight line segments
from the origin to the critical values
as in Figure \ref{fg:vp}.
See e.g.~\cite{Arnold--Gusein-Zade--Varchenko}
for vanishing cycles and the Picard-Lefschetz formula.
Let  $( \gamma_i )_{i=1}^Q$ be the corresponding distinguished basis
of vanishing cycles in $H_{N-1}(f^{-1}(0); \bZ)$.
We choose Lefschetz thimbles
$( \Gamma_i )_{i=1}^Q$
as in Figure \ref{fg:thimbles},
which gives a basis of the relative homology group
$
 H_N(\bT, \ \Re(f/z) \ll 0; \bZ)
$
for $\arg(z) > 0$.
They conjecturally correspond to the full exceptional collection
$(\scFtilde_i)_{i=1}^Q$
in the derived category of coherent sheaves on $\bP$
under homological mirror symmetry
\cite{Kontsevich_ENS98},
and Theorem \ref{th:stokes} below gives an evidence
for this conjecture.
The thimbles $(\Gamma_i')_{i=1}^Q$
shown in dotted lines are the dual Lefschetz thimbles,
which is a basis of
$
 H_N(\bT, \ \Re(f/z) \ll 0; \bZ)
$
for $\arg(z) < 0$ and
should correspond to the dual exceptional collection
$(\scEtilde_i)_{i=1}^Q$
under homological mirror symmetry.
\begin{figure}[htbp]
\begin{minipage}{.5 \linewidth}
\centering
\input{vp.pst}
\caption{Vanishing paths}
\label{fg:vp}
\end{minipage}
\begin{minipage}{.5 \linewidth}
\centering
\input{thimbles.pst}
\caption{Lefschetz thimbles}
\label{fg:thimbles}
\end{minipage}
\end{figure}
The stationary-phase integral \eqref{eq:spint}
is the Laplace transform
\begin{equation} \label{eq:Laplace_transf}
 J_i(t_1; z) = \int_{\ell_i} e^{s / z} {\Itilde_i}(t_1; s) d s
\end{equation}
of the period integral
$$
 \Itilde_i (t_1; s) = \int_{\gamma_i \subset f^{-1}(s)} \Omega / d f,
$$
where $\ell_i$ is a path on the $s$-plane
starting from a critical value
underlying the Lefschetz thimble $\Gamma_i$
and $\Omega / d f$ is the Gelfand-Leray form
on $f^{-1}(s)$.

The Stokes matrix
$(S_{ij})_{i,j=1}^Q$
is a part of the monodromy data
for the stationary-phase integrals in \eqref{eq:spint},
which is related to intersection numbers
of vanishing cycles as follows
(cf. e.g.
\cite[Section 4.1]{Dubrovin_GATFM} or
\cite[Section 5]{Ueda_SMQCCS}):
Let $(\Gamma_j^+)_{j=1}^Q$ be a basis of
$
 H_N(\bT, \ \Re(f/z) \ll 0; \bZ)
$
for $\arg(z) > 0$,
which is obtained from
the basis
$(\Gamma_{Q+1-i}')_{i=1}^Q$
of
$
 H_N(\bT, \ \Re(f/z) \ll 0; \bZ)
$
for $\arg(z) < 0$ 
by parallel transport
along a path in the upper half plane
$\{ z \in \bCx \mid \Im z \ge 0 \}$
with respect to the Gauss-Manin connection
on the relative homology bundle.
%$
% \Gamma'_{i} \mapsto \Gamma^+_{Q+1-i}.
%$
Then the Stokes matrix is given by
$$
 \Gamma^+_j = \sum_{i=1}^Q S_{ij} \Gamma_i.
$$
On the other hand,
Picard-Lefschetz formula
(see e.g. \cite{Pham_LDDC,
Ebeling_MGISCI,
Arnold--Gusein-Zade--Varchenko})
gives
$$
 \Gamma_j^+ - \Gamma_j
  =\sum_{i< j} (\gamma_i, \gamma_j) \Gamma_i,
$$
where
$
 (\gamma_i, \gamma_j) = (-1)^{N(N+1)/2} (\gamma_i \circ \gamma_j)
$
is $(-1)^{N(N+1)/2}$ times
the intersection number of vanishing cycles
$\gamma_i$ and $\gamma_j$.
This shows that
$$
 S_{ij} =
\begin{cases}
 (\gamma_i, \gamma_j) & i < j, \\
 \delta_{ij} & \text{otherwise}.
\end{cases}
$$

Simultaneous multiplication
$x_i \mapsto \alpha x_i$
by a constant $\alpha \in \bCx$ induces an isomorphism
from $f^{-1}(s)$ at $e^{t_1} = a$
to $f^{-1}( \alpha s)$ at $e^{t_1} = \alpha^Q a$,
so that the period integral $\Itilde_i(s; t_1)$ depends
only on the ratio of $s^Q$ and $e^{t_1}$;
$$
 \Itilde_i(s; t_1)
  = I_i(t), \qquad
 t = \lambda \, Q^Q \, e^{t_1} / s^Q, \qquad
 \lambda = \left. \prod_{\nu=0}^N q_\nu^{q_\nu}
      \right/ Q^{Q}.
$$
Here, the factor $Q^Q$ is chosen
so that the critical values
$s = f(p_i)$ go to $t = \lambda$.
%If one fixes $t_1$,
For any fixed value of $t_1$,
the function $\Itilde_i(s; t_1)$
is holomorphic at $s = 0$
and has singularities at $s = \infty$
and the critical values $s = f(p_i)$.
On the other hand,
the function $I_i(t)$ satisfies
the irreducible hypergeometric differential equation
$
 \scH^{\reduced} I_i(t) = 0
$
and has singularities at $t = 0$, $\lambda$ and $\infty$.
The singularities of $I_i(t)$ at $t = 0$ and $\lambda$
come from those of $\Itilde_i(s; t_1)$,
whereas the singularity at $t = \infty$
comes from the $Q$-fold Kummer covering
$t \sim 1/s^Q$.

The irreducible local system $\scL^{\reduced}$
on $\bP^1 \setminus \{ 0, 1, \infty \}$
associated with $\scH^{\reduced} I = 0$
is described
%isomorphic to the following local system
by Golyshev \cite{Golyshev_RRV}
as follows:
Let $Y$ be the anticanonical hypersurface in $\bP$
and $\iota : Y \hookrightarrow \bP$ be the inclusion.
Let further $K(\bP)$ be the Grothendieck group of $\bP$ and
$K$ be the subgroup of $K(\bP)$
generated by $\{ [\iota_* \scO_Y(i)] \}_{i \in \bZ}$.
Set
\begin{align*}
 \varphi_0(x) &= \prod_{\nu=0}^N(x^{q_\nu}-1), \quad
% \varphi_\infty(x) &= \prod_{k=1}^r(x^{d_k}-1),
 \varphi_\infty(x) = x^Q-1, \quad
 \eta(x) = \gcd \lb \varphi_0(x), \varphi_\infty(x) \rb,
\end{align*}
and
\begin{align*}
% \varphi_0(x) &= \varphibar_0(x) \cdot \eta(x), \\
% \varphi_\infty(x) &= \varphibar_\infty(x) \cdot \eta(x).
 \varphibar_\infty(x)
  = \frac{\varphi_\infty(x)}{\eta(x)}, \qquad
%  = x^{\Qred} + A'_1 x^{\Qred-1} + A'_2 x^{\Qred-2} + \dots + A'_{\Qred},
 \varphibar_0(x)
  = \frac{\varphi_0(x)}{\eta(x)},
%  = x^{\Qred} + B'_1 x^{\Qred-1} + B'_2 x^{\Qred-2} + \dots + B'_{\Qred},
\end{align*}
so that $Q = \deg(\varphi_0(x)) = \deg(\varphi_\infty(x))$ and $\Qred = Q - \deg(\eta(x))$.
Since $K(\bP)$ is generated by $\{ \scO_\bP(i) \}_{i \in \bZ}$
with relations
$$
 \sum_{k=0}^{N+1}
  (-1)^k
   \sum_{0 \le j_1 < \cdots < j_k \le N}
    \left[
     \scO_{\bP} \lb i - \sum_{\ell=1}^k q_{j_\ell} \rb
    \right]
  = 0
$$
coming from the exact sequence \eqref{eq:KoszulP},
one has a ring isomorphism
\begin{align} \label{eq:KPisom}
 K(\bP)
  \cong
 \bZ[x, x^{-1}] \left/ \varphi_0(x) \right. 
\end{align}
sending $x^i$ to $[\scO_\bP(i)]$.
%for $ \varphi_0(x) = \prod_{i=0}^N (1 - x^{q_i})$.
It follows from \eqref{eq:KoszulY} that
\begin{align} \label{eq:OYi}
 [\iota_* \scO_Y(i)]
  = [\scO_\bP(i)] - [\scO_\bP(i-Q)],
%  \sum_{k=0}^r
%     (-1)^k \sum_{1 \le j_1 < \cdots < j_k \le r}
%             \left[
%              \scO_{\bP} \lb i - \sum_{\ell=1}^k d_{j_\ell} \rb
%             \right],
\end{align}
which goes to $x^i(1-x^{-Q})$
%\begin{align*}
% \sum_{k=0}^r
%     (-1)^k \sum_{1 \le j_1 < \cdots < j_k \le r}
%            x^{i - \sum_{\ell=1}^k d_{j_\ell}}
%     &= x^i \prod_{j=1}^r \lb 1 - x^{-d_j} \rb \\
%     &= (-1)^r x^{i-\sum_{j=1}^r d_j} \prod_{j=1}^r \lb 1 - x^{d_j} \rb
%\end{align*}
under the isomorphism \eqref{eq:KPisom}.
This shows that the group $K$ is isomorphic to the subgroup
of $\bZ[x, x^{-1}] \left/ \varphi_0(x) \right.$
generated by
$
 x^i (1 - x^Q)
$
for $i \in \bZ$, and one has a short exact sequence
$$
\begin{CD}
 0 @>>> K @>>> K(\bP) @>>> K(\bP)/K @>>> 0,
\end{CD}
$$
where
\begin{align*}
 K
  \cong \varphi_\infty(x) \bZ[x, x^{-1}] / (\varphi_0(x))
  \cong \bZ[x, x^{-1}] / (\varphibar_0(x))
\end{align*}
and
\begin{align*}
 K(\bP) / K
  \cong \bZ[x, x^{-1}]/(\varphi_0(x), \varphi_\infty(x))
  = \bZ[x, x^{-1}]/(\eta(x)).
\end{align*}
Let $\Phi_0$ be an autoequivalence of $D^b \coh \bP$
defined by $\Phi_0(-) = \scO_\bP(-1) \otimes (-)$.
The induced map
$\phi_0 : K(\bP) \to K(\bP)$
is given by multiplication by $x^{-1}$,
which clearly has $\varphi_0(x)$
as the characteristic polynomial.
The autoequivalence $\Phi_0$ induces
an autoequivalence $\Phibar_0$ of $D^b \coh Y$
given by $\Phibar_0(-) = \scO_Y(-1) \otimes (-)$.
The induced map $\phibar_0 : K \to K$
fits into the commutative diagram
\begin{align*} \label{eq:diagram}
\begin{psmatrix}
 K & K(\bP) \\
 K & K(\bP)
\end{psmatrix}
\psset{shortput=nab}
\ncline[arrows=H->,hooklength=2mm,hookwidth=-2mm,offset=-1mm,nodesep=4mm]{1,1}{1,2}
\ncline[arrows=H->,hooklength=2mm,hookwidth=-2mm,offset=-1mm,nodesep=4mm]{2,1}{2,2}
\ncline[arrows=->,nodesep=5mm]{1,1}{2,1}_{\phibar_0}
\ncline[arrows=->,nodesep=4mm]{1,2}{2,2}^{\phi_0}
\end{align*}
since
\begin{align*}
 \iota_*(\scO_Y(-1) \otimes (-))
  &\cong \iota_*(\iota^* \scO_\bP(-1) \otimes (-)) \\
  &\cong \scO_\bP(-1) \otimes \iota_*(-),
\end{align*}
and the characteristic polynomial of $\phibar_0$ is given by
$\varphibar_0(x)$.

Let $\Phibar_1 = T_{\scFbar_1}^\vee$ be the autoequivalence of $D^b \coh Y$
appearing in Lemma \ref{lm:hone}.
The induced map will be denoted by $\phibar_1 : K \to K$.
Define an autoequivalence $\Phibar_\infty$ of $D^b \coh Y$ by
$$
 \Phibar_\infty = \lb \Phibar_0 \circ \Phibar_1 \rb^{-1}.
$$
The induced map $\phibar_\infty : K \to K$ is given by
$$
 \phibar_\infty = (\phibar_0 \circ \phibar_1)^{-1},
$$
which is presented by the matrix
$h_\infty = (h_1 \circ h_0)^{-1}$
by Lemmas \ref{lm:hinfty} and \ref{lm:hone}.
It follows that $\phibar_\infty$ acts on $\{ [\scFbar_i] \}_{i=1}^{Q}$
by cyclic permutation
$$
 [\scFbar_i] \mapsto \sum_{j=1}^Q (h_\infty)_{ij} [\scFbar_j]
  = [\scFbar_{i-1}].
$$
Since $\{ [\scF_i] \}_{i=1}^Q$ is a free $\bZ$-basis of $K(\bP)$,
one can define a linear automorphism
$\phi_\infty : K(\bP) \to K(\bP)$ by
\begin{equation} \label{eq:phiinfty2}
\phi_\infty([\scF_i]) = [\scF_{i-1}].
\end{equation}
%The fact that $\phi_\infty$ is defined
%without using an autoequivalence of $D^b \coh \bP$ is irrelevant
%for the discussion below.
We have a commutative diagram
\begin{equation} \label{eq:diagram2}
\begin{psmatrix}
 K & K(\bP) & K(\bP) / K \\
 K & K(\bP) & K(\bP) / K,
\end{psmatrix}
\psset{shortput=nab}
\ncline[arrows=H->,hooklength=2mm,hookwidth=-2mm,offset=-1mm,nodesep=4mm]{1,1}{1,2}
\ncline[arrows=H->,hooklength=2mm,hookwidth=-2mm,offset=-1mm,nodesep=4mm]{2,1}{2,2}
\ncline[arrows=->>,nodesep=5mm]{1,2}{1,3}
\ncline[arrows=->>,nodesep=5mm]{2,2}{2,3}
\ncline[arrows=->,nodesep=5mm]{1,1}{2,1}_{\phibar_\infty}
\ncline[arrows=->,nodesep=4mm]{1,2}{2,2}^{\phi_\infty}
\ncline[arrows=->,nodesep=4mm,linestyle=dotted]{1,3}{2,3}%^{\phi_\infty}
\end{equation}
which induces the dotted vertical arrow on the right.
By taking the transpose of the diagram \eqref{eq:diagram2},
one obtains the commutative diagram
\begin{align*} %\label{eq:diagram2}
\begin{psmatrix}
 K^\vee & K(\bP)^\vee & (K(\bP) / K)^\vee \\
 K^\vee & K(\bP)^\vee & (K(\bP) / K)^\vee.
\end{psmatrix}
\psset{shortput=nab}
\ncline[arrows=H->,hooklength=2mm,hookwidth=2mm,offset=1mm,nodesep=4mm]{1,3}{1,2}
\ncline[arrows=H->,hooklength=2mm,hookwidth=2mm,offset=1mm,nodesep=4mm]{2,3}{2,2}
\ncline[arrows=->>,nodesep=5mm]{1,2}{1,1}
\ncline[arrows=->>,nodesep=5mm]{2,2}{2,1}
\ncline[arrows=->,nodesep=5mm]{1,1}{2,1}_{\phibar_\infty^T}
\ncline[arrows=->,nodesep=4mm]{1,2}{2,2}^{\phi_\infty^T}
\ncline[arrows=->,nodesep=4mm]{1,3}{2,3}^{\left. \phi_\infty^T \right|_{(K(\bP) / K)^\vee}}
\end{align*}
Since the Euler form 
is non-degenerate by \eqref{eq:Euler_form},
it gives an identification
$K \cong K(\bP)^\vee := \Hom(K(\bP), \bZ)$
with the dual group.
The transpose $\phi_\infty^T$ of \eqref{eq:phiinfty2}
is given by the cyclic permutation
$\phi_\infty^T([\scE_i]) = [\scE_{i+1}]$ of $\{ [\scE_i] \}_{i=1}^Q$,
which goes to the cyclic permutation of $\{ x^i \}_{i=0}^{Q-1}$
under the isomorphism
$
 K(\bP) \cong \bZ[x,x^{-1}]/(\varphi_0(x)).
$
Note that the Euler form satisfies
$$
 \chi(x^j, x^i \cdot x^k)
 = \chi(\scO_\bP(j),\scO_\bP(i+k) )
 = \chi(\scO_\bP(i-j+k))
 = \chi(\scO_\bP(-i+j),\scO_\bP(k) )
  = \chi(x^{-i} \cdot x^j, x^k),
$$
so that the transpose of multiplication by $x^i$
is given by multiplication by $x^{-i}$.
One has
\begin{align*}
 (K(\bP)/K)^\vee
  &\cong K^\bot \subset K(\bP)
\end{align*}
with respect the the Euler form,
so that
\begin{align*}
 (K(\bP)/K)^\vee
  &\cong \lc \alpha \in K(\bP) \mid \chi(\alpha, \beta) = 0 \text{ for any } \beta \in K \rc \\
  &= \lc \alpha \in K(\bP) \mid \chi(\alpha, \varphi_\infty(x) \gamma) = 0
   \text{ for any } \gamma \in K(\bP) \rc \\
  &= \lc \alpha \in K(\bP) \mid \chi(\varphi_\infty(x^{-1}) \alpha, \gamma) = 0
   \text{ for any } \gamma \in K(\bP) \rc \\
  &= \Ker \lb \varphi_\infty(x^{-1}) : K(\bP) \to K(\bP) \rb \\
  &= \Ker \lb \varphi_\infty(x) : K(\bP) \to K(\bP) \rb.
\end{align*}
%Since
%$
% K(\bP) \cong \bZ[x,x^{-1}] / (\varphi_0(x)),
%$
By \eqref{eq:KPisom}
one has
\begin{align*}
 \Ker \lb \varphi_\infty(x) : K(\bP) \to K(\bP) \rb
%  &\cong \Image(\varphibar_0(x) : K(\bP) \to K(\bP)) \\
  &\cong \bZ[x,x^{-1}]/(\eta(x)).
\end{align*}
%Since $\phi_\infty^T$ acts on a basis as a cyclic permutation
%$x_i \mapsto x_{i+1}$,
Since $\phi_\infty^T$ acts on a basis as a cyclic permutation
$\phi_\infty^T([\scE_i]) = [\scE_{i+1}]$ that is the transpose of 
\eqref{eq:phiinfty2},
the characteristic polynomial of $\phi_\infty^T$ is given by $\varphi_\infty(x) = x^Q-1$.
The characteristic polynomial of the induced action
on $(K(\bP) / K)^\vee \cong \bZ[x, x^{-1}] / (\eta(x))$ is given by $\eta(x)$,
so that the characteristic polynomial of
$
 \phibar_\infty^T : K^\vee \to K^\vee
$
is given by
$
 \varphibar_\infty(x) = \varphi_\infty(x) / \eta(x).
$
Since transposition does not change the eigenvalues of a linear map,
the characteristic polynomial of $\phibar_\infty$
is also given by $\varphibar_\infty(x)$.

Let $\scK$ be the local system
associated with $K \otimes_\bZ \bC$
such that the monodromy at the origin acts by
$
 \varphibar_0 : [\iota_* \scO_Y(i)] \mapsto [\iota_* \scO_Y(i-1)]
$
and the monodromy at $\lambda$ acts by
$
 \varphibar_1 : [\iota_* \scO_Y(i)] \mapsto [\iota_* (T^\vee_{\scFbar_1} \scO_Y(i))].
$
From the explicit description of $\phi_0$ and $\phi_\infty$
in terms of the isomorphism $K \cong \bZ[x, x^{-1}]/(\varphibar_0(x))$,
one can easily see that $\phibar_0$ and $\phibar_\infty$ have
no non-trivial simultaneous invariant subspace,
so that the local system $\scK$ is irreducible.

The discussion so far shows that
the irreducible local system $\scK$ satisfies the conditions that
\begin{itemize}
 \item
the characteristic polynomial of the monodromy
at $0$ is given by $\varphibar_0(x)$,
 \item
the characteristic polynomial of the monodromy
at $\infty$ is given by $\varphibar_\infty(x)$, and
 \item
the monodromy at $\lambda$ is a pseudo-reflection
\end{itemize}
According to \cite[Theorem 3.5]{Beukers-Heckman},
an irreducible local system
on $\bP^1 \setminus \{ 0, \lambda, \infty \}$
satisfying the above conditions
is unique up to isomorphism.
One can see from \cite[Section 2]{Beukers-Heckman}
that the irreducible local system $\scL^\reduced$
defined in Section \ref{sc:introduction}
satisfies these conditions,
so that one has an isomorphism of the local systems
$\scL^{\reduced}$ and $\scK$.
The local section of $\scL^{\reduced}$
coming from the integration along the vanishing cycle $\gamma_1$
is characterized as the unique section
which is holomorphic at $\lambda$,
and as such corresponds to the section $[\scFbar_1]$
in the local system $K \otimes_\bZ \bC$.
Since the monodromy at infinity $\phibar_\infty$ acts
by cyclic permutation
$$
 [\scFbar_i] \mapsto \sum_{j=1}^Q (h_\infty)_{ij} [\scFbar_j]
  = [\scFbar_{i-1}],
$$
vanishing cycles $(\gamma_i)_{i=1}^Q$
correspond to $([\scFbar_i])_{i=1}^Q$
under this isomorphism.
Since the intersection form on vanishing cycles
is monodromy-invariant,
the Gram matrix $(\gamma_i, \gamma_j)_{ij}$
is an invariant of $\Hqd$
so that it must be proportional
to $( \chi( \scFbar_i, \scFbar_j ) )_{i, j}$
by Theorem \ref{th:invariant}.
The multiplicative constant can be fixed
from the fact that the monodromy of $\scL^{\reduced}$ at $\lambda$
is the pseudo-reflection by $\gamma_1$ and
the monodromy of $\scK$ at $\lambda$ is 
the pseudo-reflection by $[\scFbar_1]$.
If $n$ is even,
then the multiplicative constant can also be fixed by
noting that $(\gamma_1, \gamma_1) = 2 = \chi(\scFbar_1, \scFbar_1)$.
It follows that the Stokes matrix is given by
$$
 S_{ij}
  = (\gamma_i, \gamma_j)
  = \chi(\scFbar_i, \scFbar_j)
  = \chi(\scFtilde_i, \scFtilde_j)
     + (-1)^n \chi(\scFtilde_j, \scFtilde_i)
  = \chi(\scFtilde_i, \scFtilde_j)
$$
for $i < j$,
so that we have the following:

\begin{theorem} \label{th:stokes}
The Stokes matrix $(S_{ij})_{i, j=1}^Q$ for the quantum cohomology
of the weighted projective space
%$\bP(q_0, \ldots, q_N)$
is given by the Gram matrix
of the full exceptional collection $(\scFtilde_i)_{i=1}^Q$
with respect to the Euler form;
\begin{equation} \label{eq:stokes}
 S_{ij} =  \chi( \scFtilde_i, \scFtilde_j ).
\end{equation}
\end{theorem}

%This is clearly known to experts,
%c.f. \cite{Iritani_ISQCMSTO}.
This generalizes
the case of the projective space
proved in \cite{Guzzetti,
Tanabe_IHGAQCPS}
and surely known to experts
(see e.g. \cite[Remark 4.13]{Iritani_ISQCMSTO}).
The relation between
Stokes matrices and exceptional collections
originates from \cite{Cecotti-Vafa_OCNST}
and was developed by Kontsevich \cite{Kontsevich_ENS98},
Zaslow \cite{Zaslow}
and Dubrovin \cite{Dubrovin_GATFM}.

%It is interesting to note that
%although the Stokes matrix is a $Q \times Q$ matrix,
%for the proof of Theorem \ref{th:stokes},
%one only needs to know the monodromy of the irreducible local system
%$\scL^\reduced$ of rank $\Qred$,
%which is defined by the hypergeometric differential operator $\scH^\reduced$.
%From a geometric point of view,
The rank $Q$ of the quantum differential equation \eqref{eq:qde}
is the rank of the relative homology group
$
 H_N(\bT, \ \Re(f/z) \ll 0; \bZ),
$
whereas the rank $\Qred$ of the irreducible local system $\scL^\reduced$
is the rank of the homology group
$
 H_{N-1}(\overline{f^{-1}(s)}; \bZ)
$
of a smooth compactification $\overline{f^{-1}(s)}$
of the fiber $f^{-1}(s)$.
As a result,
sections of the reducible local system $\scL$
not coming from sections of $\scL^\reduced$
are not period integrals
on the mirror manifold $\overline{f^{-1}(s)}$.
Since the stationary-phase integrals $J_i$ are
Laplace transform of the period integrals $\Itilde_i$
as described in \eqref{eq:Laplace_transf},
one only needs the monodromy
of the irreducible local system
$\scL^\reduced$
for the proof of Theorem \ref{th:stokes}.
%Nevertheless,
%Theorem \ref{th:monodromy-group} states that
%the monodromy of not only $\scL^\reduced$
%but also $\scL$ are related
%to the structure of $D^b \coh Y$.
%%(or the images $\scFbar_i$ of generators $\scF_i$ of $D^b \coh \bP$,
%%to be more precise).

\bibliographystyle{amsalpha}
\bibliography{bibs}

\noindent
Susumu Tanab\'e

\noindent
Department of Mathematics,\\
Galatasaray University\\
\c{C}{\i}ra$\rm\breve{g}$an cad. 36,\\
Be\c{s}ikta\c{s}, Istanbul, 34357,
Turkey

{\em e-mail address}\ : \ tanabesusumu@hotmail.com, tanabe@gsu.edu.tr

\ \\

\noindent
Kazushi Ueda

\noindent
Department of Mathematics,\\
Graduate School of Science,\\
Osaka University,\\
Machikaneyama 1-1,
Toyonaka,
Osaka,
560-0043,
Japan.

{\em e-mail address}\ : \  kazushi@math.sci.osaka-u.ac.jp

\end{document}